\documentclass[11pt]{article}

\usepackage[absolute,overlay]{textpos}

\usepackage{float}
\usepackage[dvips]{graphicx}
\usepackage{psfrag}
\usepackage{amsmath,amsthm,amsfonts,amssymb,amscd}
\usepackage[table]{xcolor}
\usepackage{mathrsfs}
\usepackage{upgreek}
\usepackage{stackrel}
\usepackage{tipa}
\usepackage{accents}
\usepackage[multiple]{footmisc}
\usepackage{multirow} 
\usepackage{booktabs}
\numberwithin{equation}{section}
\usepackage{multicol}
\definecolor{SmartBlue}{RGB}{51, 51, 255}
\usepackage{hyperref}
\hypersetup{filecolor=magenta}
\hypersetup{colorlinks=true}
\hypersetup{urlcolor=webbrown}
\hypersetup{linkcolor=SmartBlue}
\hypersetup{pdfpagemode=UseNone}
\hypersetup{citecolor=SmartBlue}

\usepackage[absolute,overlay]{textpos}

\usepackage{changepage}

\usepackage{float}

\usepackage{amsmath,amsthm,amsfonts,amssymb,amscd}
\usepackage{mathrsfs}
\usepackage{upgreek}
\usepackage{stackrel}
\usepackage{tipa}
\usepackage{accents}
\usepackage[multiple]{footmisc}

\usepackage{multirow} 
\usepackage{booktabs}

\usepackage{xfrac} 

\newtheorem{theorem}{Theorem}[section]
\newtheorem{proposition}[theorem]{Proposition}
\newtheorem{lemma}[theorem]{Lemma}
\newtheorem{definition}[theorem]{Definition}

\newtheorem{remark}[theorem]{Remark}
\newtheorem{example}[theorem]{Example}

\newtheorem*{orient*}{Orientation Condition}
\newtheorem*{junction*}{Junction Conditions}

\def\Mkp{\mathbb{M}^+}
\def\Mkm{\mathbb{M}^-}
\def\Mkpm{\mathbb{M}^{\pm}}
\def\flatp{\eta^+}
\def\flatm{\eta^-}
\def\flatpm{\eta^{\pm}}

\def\hatv{\hat{v}}
\def\hatz{\hat{z}}
\def\supp{\textup{supp}}

\def\qfunction{q}

\newcommand{\cx}{\mathcal{X}}

\newcommand{\heav}{\theta(u)}
\newcommand{\up}{u_{+}(u)}

\def\defi{{\stackrel{\mbox{\tiny {\textbf{def}}}}{\,\, = \,\, }}}

\def\metdata{\{\mathcal{N},\gamma,\ellc,\elltwo\}}

\def\fv{\epsilon}

\def\n{\mathfrak{n}}

\def\Fcal{\mathcal F}

\def\bY{\textup{\textbf{Y}}}

\def\Y{\textup{Y}}

\newcommand{\nn}{\nonumber}

\newcommand{\bm}[1]{\mbox{\boldmath $#1$}}

\newcommand\N{\mathcal N}
\newcommand\M{\mathcal M}

\newcommand\elltwo{\ell^{(2)}}

\newcommand\hypdata{\{ \mathcal{N},\gamma,\ellc, \elltwo, \bY\}}
\newcommand\rig{\zeta}

\newcommand\A{\mathcal A}

\def\bmell{\bm{\ell}}
\def\ellc{\bmell}

\def\defi{{\stackrel{\mbox{\tiny \textup{\textbf{def}}}}{\,\, = \,\, }}}

\def\nabh{\nabla^h}

\def\metdata{\{\mathcal{N},\gamma,\ellc,\elltwo\}}

\def\n{\mathfrak{n}}

\def\Fcal{\mathcal F}

\def\bY{\textup{\textbf{Y}}}

\def\Y{\textup{Y}}

\def\bmell{\bm{\ell}}
\def\ellc{\bmell}

\newcommand{\ov}{\overline}
\newcommand{\nfi}{\varphi}
\newcommand{\cp}{\partial}

\newcommand{\eps}{\varepsilon}
\newcommand{\bs}{\boldsymbol}
\newcommand{\lp}{\left(}
\newcommand{\rp}{\right)}

\newcommand{\cu}{\mathcal{U}}
\newcommand{\cv}{\mathcal{V}}
\newcommand{\Mp}{\mathcal{M^+}}
\newcommand{\Ml}{\mathcal{M^-}}

\newcommand{\lb}{\left\lbrace}
\newcommand{\rb}{\right\rbrace}
\newcommand{\lc}{\left[}
\newcommand{\rc}{\right]}
\newcommand{\ld}{\left.}
\newcommand{\rd}{\right.}

\newcommand{\rv}{\right\vert}
\newcommand{\la}{\langle}
\newcommand{\ra}{\rangle}

\newcommand{\Mpm}{\mathcal{M}^{\pm}}

\newcommand{\wt}{\widetilde}

\newcommand{\tcte}{}

\newcommand{\nullhyp}{\widetilde{\N}}

\newcommand{\spc}{\textup{ }}

\usepackage[new]{old-arrows}

\makeatletter
\newsavebox\myboxA
\newsavebox\myboxB
\newlength\mylenA

\newcommand*\xoverline[2][0.75]{%
    \sbox{\myboxA}{$\m@th#2$}%
    \setbox\myboxB\null
    \ht\myboxB=\ht\myboxA%
    \dp\myboxB=\dp\myboxA%
    \wd\myboxB=#1\wd\myboxA
    \sbox\myboxB{$\m@th\overline{\copy\myboxB}$}
    \setlength\mylenA{\the\wd\myboxA}
    \addtolength\mylenA{-\the\wd\myboxB}%
    \ifdim\wd\myboxB<\wd\myboxA%
       \rlap{\hskip 0.5\mylenA\usebox\myboxB}{\usebox\myboxA}%
    \else
        \hskip -0.5\mylenA\rlap{\usebox\myboxA}{\hskip 0.5\mylenA\usebox\myboxB}%
    \fi}
\makeatother

\newcommand{\hor}{\mathscr{H}}

 \newcounter{mnotecount}

 \newcommand{\mnote}[1]
 {\protect{\stepcounter{mnotecount}}$^{\mbox{\tiny
 $\,\bullet$\themnotecount}}$ \marginpar{
 \raggedright\tiny\em
 $\,\bullet$\themnotecount: #1} }

\allowdisplaybreaks

\usepackage[a4paper,top=25mm,bottom=25mm,left=25mm, right=25mm,bindingoffset=0mm]{geometry}

\title{Generalizing the Penrose cut-and-paste method:\\ Null shells with pressure and energy flux
}

\author{
	Miguel Manzano$^1$\thanks{{\tt m.manzano.rod@gmail.com}}\ ,
	Argam Ohanyan$^2$\thanks{{\tt argam.ohanyan@utoronto.ca}}\ \ and Roland Steinbauer$^1$\thanks{{\tt roland.steinbauer@univie.ac.at}} \\ \\
	$^1$ Faculty of Mathematics, University of Vienna, \\
	Oskar-Morgenstern-Platz 1, 1090 Vienna, Austria. \\ \\
    $^2$ Department of Mathematics, University of Toronto,\\ 45 St.\ George Street, M5S 2E5 Toronto, Ontario, Canada. 
}

\usepackage{bbm}

\usepackage{tocloft}

\begin{document}

\maketitle

\begin{abstract}
The cut-and-paste method is a procedure for constructing null thin shells by matching two regions of the same spacetime across a null hypersurface. Originally proposed by Penrose, it has so far allowed to describe purely gravitational and null-dust shells  in constant-curvature backgrounds.\ In this paper, we extend the cut-and-paste method to null shells with arbitrary gravitational/matter content. To that aim, we first derive a locally Lipschitz continuous form of the metric of the spacetime resulting from the most general matching of two constant-curvature spacetimes with totally geodesic null boundaries, and then obtain the coordinate transformation that turns this metric into the cut-and-paste form with a Dirac-delta term. The paper includes an example of a null shell with non-trivial energy density, energy flux and pressure in Minkowski space.
\end{abstract}

\tableofcontents

\section{Introduction}\label{sec:intro}

In nature, it is common to find neighbouring spacetime regions with different matter content, for instance the interior and exterior of a neutron star. Metric theories of gravity must hence be capable of matching spacetimes, i.e.\ of constructing one single spacetime by suitably gluing regions with distinct matter-energy properties.\ In particular, null matchings, i.e.\ matchings across a null hypersurface, are of special relevance in General Relativity 
since they are key to describing gravitational waves, as well as black holes and gravitational collapse.\ 
As a result, countless examples of null shells have been studied in various scenarios and contexts, see e.g.\ \cite{barrabes2003singular, khakshournia2023art, poisson2004relativist} and references therein.  
\medskip

There exist two approaches to null matchings:\ 
the \textit{standard matching theory} going back to Darmois \cite{darmois1927memorial} in the timelike/spacelike cases and later generalized to the null \cite{barrabes1991thin}
and to the arbitrary causal character cases  \cite{mars2007lorentzian,mars1993geometry}, and
Penroses's so-called \textit{``cut-and-paste" method} \cite{Penrose:1972xrn}.
\medskip

In the context of the former, particularly relevant to this paper are the recent works \cite{manzano2021null,manzano2022general,manzano2023matching}. Based on the 
\textit{formalism of hypersurface data} \cite{mars2013constraint,mars2020hypersurface,mars2024abstract}, they provide a \textit{completely general} approach---that is, without 
a-priori assumptions on the matching hypersurface or the spacetimes to be matched---which at the same time allows to obtain \textit{very explicit} results, e.g.\ the 
gravitational/matter content of \textit{any} null shell, or the first example of a null shell \textit{with pressure} in Minkowski space.\ It is worth emphasizing, however, that \cite{manzano2021null,manzano2022general,manzano2023matching} concentrate primarily on identifying the geometric objects that encode the full information of the matching  
and on deriving the expression for the energy-momentum tensor of the null shell, rather than on obtaining an explicit form of the metric of the matched spacetime.
\medskip

Alternatively Penrose has put forward his vivid cut-and-paste method \cite{Penrose:1972xrn}.\ Originally, he used it to construct impulsive $pp$-waves propagating in Minkowski space by cutting it into two ``halves'' along a null hyperplane $\nullhyp=\{\cu=0\}$, and reattaching them with a shift along the null generators of $\nullhyp$. This identification of boundary points with a jump along the null direction are called the Penrose junction conditions. The {matched} spacetime is then described by a metric with a Dirac $\delta$-term supported on $\nullhyp$,
\begin{equation}
	\label{CAP:Mk:metric} ds^2=-2d\cu d\cv+d\cx^2+d\mathcal{Y}^2+2\mathcal{H}(\cx,\mathcal{Y})\delta(\cu)d\cu^2,
\end{equation}
clearly displaying curvature concentrated on the single wave surface $\nullhyp$. In fact, this metric lies beyond the largest general class of metrics that enables a stable definition of (distributional) curvature \cite{geroch1987strings,lefloch2007definition,SV:09}, but the high symmetry nevertheless allows one to calculate the curvature explicitly.\ Penrose also has constructed expanding, spherical impulsive waves in Minkowski space, and later on the method has been extended to all constant-curvature backgrounds, see \cite[Ch.\ 20]{griffiths2009exact}, and also \cite{PS:22} for a recent pedagogical review.
\medskip

In his seminal work \cite{Penrose:1972xrn}, Penrose also has introduced a continuous form of the metric \eqref{CAP:Mk:metric}, which was clearly ``physically equivalent", but the precise relation between the two metrics remained obscure. In later works the geometric meaning of the notorious ``discontinuous coordinate transformation" relating the two metric forms has been made transparent \cite{S:98}, and formulated in a mathematically rigorous way \cite{KS:99}, recently even with a non-zero cosmological constant \cite{podolsky2019cut,samann2023cut}.

An important limitation of the cut-and-paste method is that, so far, it
has \textit{only} allowed for the construction of purely gravitational or null-dust shells.\  However, null shells with much more general matter contents have been built in \cite{manzano2021null,manzano2022general,manzano2023matching}---both in constant curvature backgrounds and in more general settings---, and in \cite{manzano2021null} a firm connection to the cut-and-paste method has been established. In this work we generalize the cut-and-paste method for all constant-curvature spacetimes to produce null shells with \textit{completely general} energy-momentum tensors.\ We also obtain explicit expressions for the corresponding metric forms, which encode the matter content of the shell in a transparent way in the metric functions.
\medskip

In particular, we first derive a locally Lipschitz continuous form of the metric $g$ of the spacetime $(\M,g)$ resulting from the \textit{most general} matching of two regions of Minkowski or (anti-)de Sitter space across a totally geodesic null hypersurface (Sections \ref{sec:lip:metric} and \ref{sec:AdS}). In this context, we also revisit the general matching of semi-Riemannian manifolds (Appendix \ref{app:B}) and prove that under 
natural mild conditions the metric of the matched manifold is locally Lipschitz (w.r.t.\ its natural smooth differentiable structure), thereby straightening an old result of \cite{clarke1987junction}. 
\medskip

Then we also derive a vivid distributional form of the metric, again displaying a Dirac-$\delta$ term supported on the matching hypersurface. To do so, we first introduce a ``discontinuous coordinate transformation" together with a distributional ansatz for $g$, and prove that the ansatz and the Lipschitz continuous metric forms are related by the proposed coordinate transformation (Sections \ref{sec:dist:metric} and \ref{sec:AdS}). Here we defer some technical calculations regarding regularizations of distributional products to Appendix \ref{app:distributional:identities}. While our treatment here cannot be considered final, we firmly prepare the ground for future studies of the geodesics in these distributional geometries, which we expect in analogy to the classical cut-and-paste case to open the gates to a full nonlinear distributional analysis of the notorious transformation, cf.\ \cite{ SSLP:16,SS:17,podolsky2019cut,samann2023cut}.
\medskip

Our results not only extend the cut-and-paste construction to generic null shells in spacetimes with constant curvature, but also yield the corresponding generalized Penrose junction conditions, which determine the shift along the null generators of the matching hypersurface. Finally, we include an explicit example of a null shell with non-trivial energy density, energy flux and pressure in Minkowski space (Section \ref{sec:ex:pressure:Mk:new}).

\subsection{Notation and conventions}
All manifolds are smooth and connected.\ Given a manifold $\M$ we use $\mathcal{F}\lp\mathcal{M}\rp\defi C^{\infty}\lp\mathcal{M},\mathbb{R}\rp$, and $\mathcal{F}^{\star}\lp\mathcal{M}\rp\subset\mathcal{F}\lp\mathcal{M}\rp$ for the subset of no-where zero functions.\  The tangent bundle is denoted 
by $T\mathcal{M}$, and $\Gamma\lp T\mathcal{M}\rp$ is the set
of sections (i.e.\ vector fields).\  We use $\pounds$, $d$ for the Lie derivative and exterior derivative.\ Given a differentiable
map $\phi$ between manifolds, we use the standard notation of $\phi^{\star}$ and $\phi_{\star}$ for
its pull-back and push-forward respectively.\ Both tensorial and abstract index notation will be employed.\ We work in arbitrary dimension $\mathfrak{n}$ and use the sets of indices
\begin{equation}
	\label{notation}
	\alpha,\beta,...=0,1,2,...,\mathfrak{n};\qquad a,b,...,i,j,...=1,2,...,\mathfrak{n};\qquad A,B,...,I,J,...=2,...,\mathfrak{n},
\end{equation}
\tcte{irrespectively of whether they are tensorial or identify elements in a set}.\ When index-free notation is used (and only then), we shall distinguish
covariant tensors with boldface.\ As usual, parenthesis denote symmetrization of indices.\ Our signature convention for Lorentzian manifolds $\lp \mathcal{M},g\rp$ is $(-,+, ... ,+)$, and we use $\nabla$ to denote the Levi-Civita derivative of $g$.

\section{Spacetime matching and the hypersurface data formalism}\label{sec:matching:2021}

In this section we recall the necessary facts on the null matching of spacetimes using the formalism of hypersurface data, with special emphasis on the results in \cite{manzano2021null,manzano2022general, manzano2023matching}.

Originally put forward in \cite{mars2013constraint,mars2020hypersurface} (see also \cite{manzano2023field,manzano2023constraint,mars2024transverseI,mars2024transverseII}), the formalism of hypersurface data relies on the notions of \textit{metric hypersurface data} and \textit{hypersurface data}, which we now introduce specifically in the null case. 

Let $\N$ be a smooth and connected $\n$-manifold endowed with a symmetric $(0,2)$-tensor $ \gamma $ \tcte{with one-dimensional radical (i.e.\ with signature $(0,+,...,+)$)}, a one-form $\ellc$ and a scalar function $\ell^{(2)}$.\ The tuple $\metdata$ defines \textit{null metric hypersurface data} provided that the square $(\n+1)$-matrix
\begin{equation}
	\bs{\A}\defi \lp \hspace{-0.1cm}
	\begin{array}{cc}
		\gamma_{ab} & \ell_a\\
		\ell_b & \elltwo
	\end{array}
	\hspace{-0.1cm}\rp
\end{equation}
is non-degenerate everywhere on $\N$.\ \textit{Null hypersurface data}, on the other hand, is a tuple $\lb \N, \gamma ,\ell,\ell^{(2)},\bY\rb$ consisting of null metric hypersurface data $\metdata$ equipped with an additional symmetric $(0,2)$-tensor $\bY$.

To connect these concepts with the geometry of embedded null hypersurfaces one uses the notion of \textit{embeddedness}. A null metric data $\metdata$ is said to be $\{\phi,\rig\}$-embedded in an $(\n+1)$-spacetime $\lp \mathcal{M},g\rp$ if there exists an embedding $\phi:\N\longhookrightarrow\mathcal{M}$ and
a vector field $\rig$ along $\phi(\N)$---called \textit{rigging}---everywhere transversal to $\phi(\N)$, that satisfy
\begin{equation}
	\label{emhd}
	\phi^*\lp g\rp= \gamma , \qquad\phi^*\lp g\lp\rig,\cdot\rp\rp=\ellc, \qquad\phi^*\lp g\lp\rig,\rig\rp\rp=\ell^{(2)}.
\end{equation}
For hypersurface data $\hypdata$, embeddedness requires the additional condition 
\begin{equation}
	\label{eqcf}
	\dfrac{1}{2}\phi^*\lp \pounds_{\rig}g\rp=\bY.
\end{equation}
When the data is embedded, it follows from \eqref{emhd}-\eqref{eqcf} that $\gamma$ coincides with the first fundamental form of the hypersurface, while $\ellc,\elltwo$ codify the transverse components of $g$.\ Thus, $\gamma,\ellc,\elltwo$ capture the full spacetime metric $g$ along $\phi(\N)$ and $\bY$ codifies first transverse derivatives of $g$ at $\phi(\N)$.

Now to proceed to the spacetime matching, let us consider two $(\n+1)$-dimensional spacetimes $(\mathcal{M}^\pm, g^\pm)$ with null boundaries $\nullhyp^{\pm}$.\ For $(\Mpm,g^{\pm})$ to be matchable across $\nullhyp^{\pm}$, they must fulfil the so-called  
\textit{junction conditions}, which have been formulated in various equivalent ways in the literature, see e.g.\  \cite{barrabes1991thin,bonnor1981junction,clarke1987junction,darmois1927memorial,israel1966singular}.\ In the language of hypersurface data they require  \cite{mars2013constraint} the existence of a null metric data $\metdata$, two embeddings $\phi^{\pm}:\N\longhookrightarrow\Mpm$ and two riggings $\rig^{\pm}$ (cf.\ Figure \ref{Fig:1}), such that $\metdata$ can be $\{\phi^{\pm},\rig^{\pm}\}$-embedded in both $(\Mp,g^+)$ and $(\Ml,g^-)$
so that
\hypertarget{junc-cond-FHD}{ } 
\begin{itemize}\itemsep0cm
\item[$(i)$]
$\phi^{\pm}(\N)=\nullhyp^{\pm} $, and
\item[$(ii)$]
one of the riggings $\rig^{\pm}$ points inwards and the other outwards with respect to $\Mpm$.\
\end{itemize}
The metric data
$\metdata$ and the pairs $\{\phi^{\pm},\rig^{\pm}\}$ must therefore satisfy (cf.\ \eqref{emhd})
\begin{align}
\label{junctcondAbs2} \gamma =(\phi^{\pm})^{\star}(g^{\pm}),\qquad
\ellc =(\phi^{\pm})^{\star}(g^{\pm}(\rig^{\pm},\cdot)),\qquad \elltwo =(\phi^{\pm})^{\star}(g^{\pm}(\rig^{\pm},\rig^{\pm})).
\end{align}

\begin{figure}[t!]
\centering
\includegraphics[angle=-90,width=14cm]{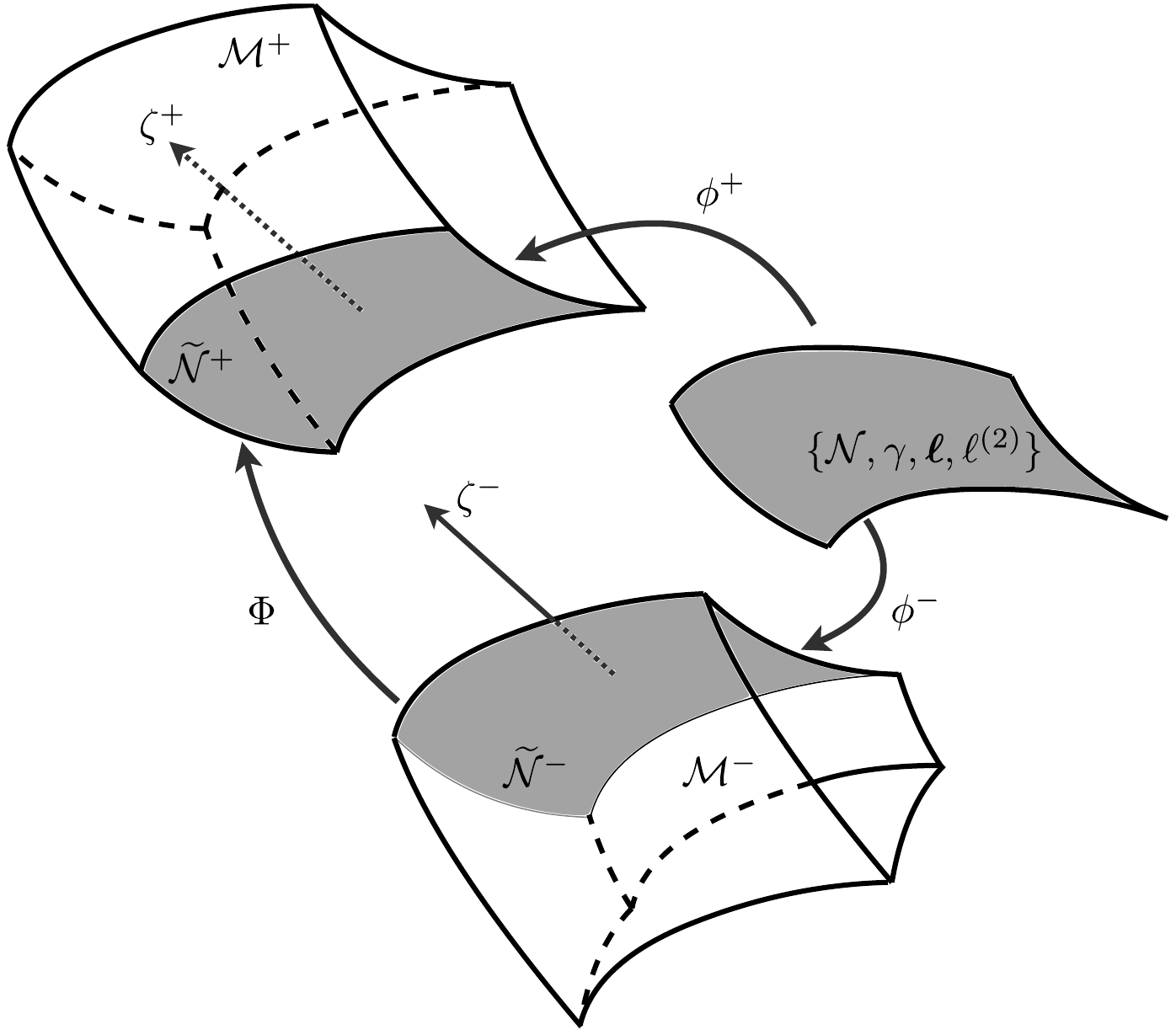}
\caption{Matching of two spacetimes $(\Mpm,g^{\pm})$ with null boundaries $\nullhyp^{\pm}$,
where $\N$ is a detached hypersurface embedded in $(\Mpm,g^{\pm})$ with embeddings $\phi^{\pm}$, $\rig^{\pm}$ are matching riggings and $\Phi$ is the corresponding matching map.}
\label{Fig:1}
\end{figure}

As discussed in \cite{mars2013constraint}, this viewpoint allows to define the energy momentum tensor of the shell in a way \textit{detached} from any ambient spacetime.\ This is essential when the spacetimes to be matched are yet to be constructed. The purely gravitational and the matter content of the shell are encoded by the jump in the extrinsic geometry of $\nullhyp^{\pm}$, namely by a tensor field $[\bY]$ defined by (cf.\ \eqref{eqcf})
\begin{equation}
\label{jump:Y}[\bY]\defi \bY^+-\bY^-,\quad\textup{where}\quad \bY^{\pm}
\defi
\frac{1}{2}(\phi^{\pm})^{\star}\lp\pounds_{\rig^{\pm}}g^{\pm}\rp.
\end{equation}
By adapting $\N$ to one of the boundaries, one can always choose the corresponding embedding freely and so w.l.o.g.\ we fix $\phi^-$. 

Since we are interested in analysing properties of matched spacetimes, it is convenient to write the junction conditions from a spacetime viewpoint as well (see e.g.\ \cite{mars1993geometry}).\ Given two spacetimes $(\Mpm, g^{\pm})$ with null boundaries $\nullhyp^{\pm}$, the matching is possible if and only if there exist two riggings $\rig^{\pm}$ along $\nullhyp^{\pm}$ and a diffeomorphism ${\Phi}:\nullhyp^- \longrightarrow \nullhyp^+$ such that, for all $X,W\in\Gamma(T\nullhyp^-)$,
\begin{align} 
\label{junctcondST:1} g^-(X,W)&=({\Phi}^{\star}g^+)(X,W),\\ 
\label{junctcondST:2} g^-(\rig^-,X)&={\Phi}^{\star}\big(g^+(\rig^+,\cdot)\big) (X),\\ 
\label{junctcondST:3} g^-(\rig^-,\rig^-)&={\Phi}^{\star}(g^+(\rig^+,\rig^+)).
\end{align}
Furthermore, the riggings $\rig^{\pm}$ must fulfil
the orientation condition \hyperlink{junc-cond-FHD}{$(ii)$} above. We then clearly have $\Phi \circ \phi^- = \phi^+$, and we shall refer to $\Phi$, $\phi^{\pm}$ and $\rig^{\pm}$ as \textit{matching map}, \textit{matching embeddings} and \textit{matching riggings} respectively, see also Figure \ref{Fig:1}. 

A key aspect is that, 
if $\rig^{\pm}$ are matching riggings verifying \eqref{junctcondST:1}-\eqref{junctcondST:3}, then any other pair 
\begin{equation}
\nn \Bigl\{\wt{\rig}^-=z(\rig^-+V),\spc\wt{\rig}^+=\widehat{z}(\rig^++\widehat{V})\Bigr\},
\end{equation}
where $ z\in\Fcal^{\star}(\nullhyp^-), V\in\Gamma(T\nullhyp^-)$ and $\widehat{z}\defi ({\Phi}^{-1})^{\star}z,\widehat{V}\defi {\Phi}_{\star}V$, 
fulfils conditions \eqref{junctcondST:1}-\eqref{junctcondST:3} as well. 
This allows one to fix one of the riggings freely and so we w.l.o.g.\ fix $\rig^-$ at our convenience.\ It is also worth stressing that, given a matching map $\Phi$ satisfying \eqref{junctcondST:1} together with a choice of $\rig^-$, \eqref{junctcondST:2}-\eqref{junctcondST:3} provide \textit{at most one} solution for $\rig^+$  \cite[Lem.\ 3]{mars2007lorentzian}.\ The matching is then possible only if such a solution exists and, \textit{in addition}, its orientation fulfils the previous requirement \hyperlink{junc-cond-FHD}{$(ii)$}.\  Thus, 
conditions \eqref{junctcondST:1}-\eqref{junctcondST:3} are necessary but not sufficient to guarantee the matching. Moreover, when the matching is possible, all the information is encoded in $\Phi$ (equivalently in $\phi^+$, once $\phi^-$ has been fixed), since one rigging can be chosen at will and the other is determined uniquely from \eqref{junctcondST:2}-\eqref{junctcondST:3} and \hyperlink{junc-cond-FHD}{$(ii)$}.\

For our purpose it is convenient to adopt the notation and setup of \cite{manzano2021null}, which we introduce next. We assume that $\nullhyp^{\pm}$ have product structure $S^{\pm} \times \mathbb{R}$, where $S^{\pm}$ is a spacelike cross-section of $\nullhyp^{\pm}$ and the null generators of $\nullhyp^{\pm}$ are along $\mathbb{R}$. 
Without loss of generality \cite{manzano2021null}, \textit{we suppose that   
$\nullhyp^-$ lies in the future of its spacetime $\Ml$ and $\nullhyp^+$  
lies in the past of $\Mp$}.\ We also let $\{ L^{\pm}, k^{\pm}, v_I^{\pm} \}$ be two bases  of $\Gamma\lp T\mathcal{M}^{\pm} \rp\vert_{\nullhyp^{\pm}}$ with the following properties (cf.\ Figure \ref{Fig:2}):
\begin{align}
\nn \text{(A)}\quad & k^{\pm} \text{ are future affine}\footnotemark[1] \text{ null generators of } \nullhyp^{\pm}. \nonumber\\[5pt]
\label{cond:ABC} \text{(B)}\quad & \{v_I^{\pm}\} \in \Gamma(T\nullhyp^{\pm}) \text{ satisfy:} \quad v_I^{\pm}\vert_{S^{\pm}} \in \Gamma(TS^{\pm}),  
\quad [k^{\pm},v_I^{\pm}] = 0,\quad [v_I^{\pm},v_J^{\pm}] = 0. \\[5pt]
\nn \text{(C)}\quad & L^{\pm} \text{ are past rigging vector fields verifying}
\textup{:}  
\quad g^{\pm}(L^{\pm},k^{\pm}) = 1, \quad g^{\pm}(L^{\pm},v_I^{\pm}) = 0.
\end{align}
\footnotetext[1]{For null hypersurfaces admitting a cross-section, there always exists an affine null generator \cite{gourgoulhon20063+}.}

\setcounter{footnote}{1}

\vspace{-0.8cm}

\begin{remark}
The conditions $g^{\pm}(L^{\pm},k^{\pm}) = 1$ and $g^{\pm}(L^{\pm},v_I^{\pm}) = 0$ were not imposed in \textup{\cite{manzano2021null,manzano2022general,manzano2023matching}}, and are not necessary for the construction.\ However, it is convenient to enforce them in order to simplify the discussion.\ Observe also that $L^{\pm}$ need not be null.
\end{remark}
To exploit the product structure let $v_{\pm}\in\Fcal(\nullhyp^{\pm})$ be the (unique) solutions of $k^{\pm}(v_{\pm})\vert_{\nullhyp^{\pm}}=1$, $v_{\pm} \vert_{S^{\pm}}=0$. Then $v_{\pm}$ define 
a foliation of $\nullhyp^{\pm}$ by a family of diffeomorphic spacelike cross-sections $\{S^{\pm}_{v_{\pm}}\}$ with induced metric $h^{\pm}$, see Figure \ref{Fig:2}. 
Note that, as a consequence of (B), the vectors $\{v_I^{\pm}\}$ are spacelike everywhere on $\nullhyp^+$, and in fact tangent to the leaves $\{S^{\pm}_{v_{\pm}}\}$.\ 
Hence, 
the components of $h^{\pm}$ in the basis $\{v_A^{\pm}\}$ are given by 
\begin{equation}
h^{\pm}_{AB}\defi g^{\pm}(v_A^{\pm},v_B^{\pm}),
\end{equation}
and we write $h_{\pm}^{AB}$ for the components of its inverse w.r.t.\ the corresponding dual basis.

To formulate the matching most explicitly we write it in terms of two bases $\{e_a^{\pm}\}$ of $\Gamma(T\nullhyp^{\pm})$ that are to be identified via $\Phi$ (i.e.\ $\{\Phi_{\star}e_a^- = e_a^+\}$). 
Suppose that $(\Mpm,g^{\pm})$ are matchable and consider coordinates $\{z^1=v,z^A\}$ on $\N$.\ Define 
vector fields $\{e_a^-\defi \phi^-_{\star}(\cp_{z^a})\}$ (cf.\ Figure \ref{Fig:2}), where $\phi^-:\N\longhookrightarrow\nullhyp^-\subset\Ml$, and enforce the following equalities on $\nullhyp^-$:
\begin{equation}
\label{eiyl}
e^-_1=k^-,\qquad e^-_I=v^-_I,\qquad \rig^-=L^-.
\end{equation}

\begin{figure}[t!]
\begin{minipage}[t]{0.5\textwidth}
  \includegraphics[width=\linewidth]{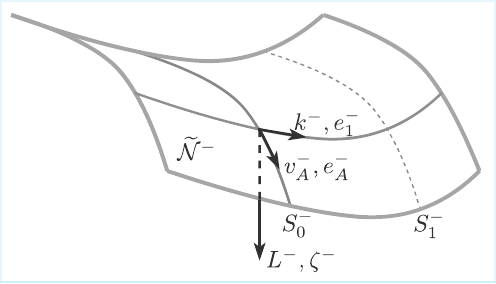}
\end{minipage}
\hfill
\begin{minipage}[t]{0.5\textwidth}
  \vspace{-4.5cm} 
  \includegraphics[width=\linewidth]{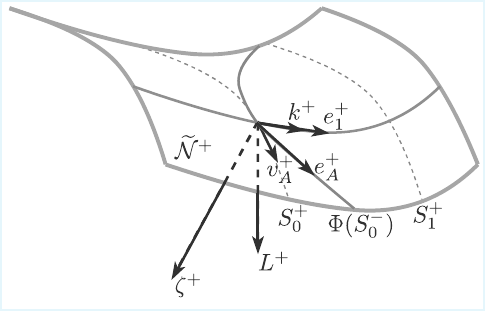}
\end{minipage}
\caption{The bases $\{L^{\pm},k^{\pm},v^{\pm}_A\}$ and $\{\rig^{\pm},e_1^{\pm},e^{\pm}_A\}$ of $\Gamma(T\Mpm)\vert_{\nullhyp^{\pm}}$, the families of spacelike cross-sections $\{S^{\pm}_{v_{\pm}}\}\subset\nullhyp^{\pm}$, and the image $\Phi(S_0^-)\subset\nullhyp^+$ of one such cross-sections via the matching map $\Phi$. In our setup, $\M^-$ lies ``below'' $\nullhyp^-$ and $\M^+$ ``above'' $\nullhyp^+$, hence the matching riggings have suitable orientations.}
\label{Fig:2}
\end{figure}

By the junction conditions, 
there exist $\{\phi^+,\rig^+\}$ so that \eqref{junctcondAbs2} holds and the orientations of $\rig^{\pm}$ satisfy \hyperlink{junc-cond-FHD}{$(ii)$}. Now employing the basis $\{e_a^+\defi \phi^+_{\star}(\cp_{z^a})\}$ of $\Gamma(T\nullhyp^+)$ (see Figure \ref{Fig:2}), determining the matching (i.e., finding $\phi^+$) reduces to deriving explicit expressions for $\{e_a^{+}\}$.\ 
These have been computed explicitly in \cite{manzano2021null} to 
\begin{equation}
\label{eqA4}
e^+_1=\dfrac{\cp H(v,z^A)}{\cp v}k^+,\qquad e^+_I=\dfrac{\cp H(v,z^A)}{\cp z^I}k^++\dfrac{\cp h^J( z^A)}{\cp z^I}v^+_J,
\end{equation}
where $H(v,z^B)$, $h^A(z^B)$ are scalar functions on $\N$ with the following properties:\

\vspace{-0.33cm}

\begin{itemize}\itemsep0cm
\item[$(a)$]  
$H$ satisfies  $\cp_{v}H>0$,  
\item[$(b)$] the Jacobian matrix $\frac{\cp (h^2,...,h^{\n+1})}{\cp(z^2,...,z^{\n+1})}$ has non-zero determinant, and 
\item[$(c)$] the functions $\{h^A\}$ verify 
\begin{align}
	\label{junct1}h^-_{AB}\vert_{\phi^-(p)} &=\frac{\cp h^C}{\cp z^A} \frac{\cp h^D}{\cp z^B}h^+_{CD}\vert_{\phi^+(p)}\quad\forall p\in\N.
\end{align}
\end{itemize}
\begin{remark} 
The combination of \eqref{eiyl} and (A)-(C) implies that $\rig^-$ is past and $v$ is a coordinate along the degenerate direction of $\N$.\ However, 
\eqref{eiyl} still enables to select the pair $\{\phi^-,\rig^-\}$ at will, as the restrictions (A)-(C) allow for some freedom in the choice of $\{L^-,k^-,v^-_I\}$.\ 
\end{remark}
As discussed in \cite{manzano2021null,manzano2023matching}, the core issue for the feasibility of a matching lies precisely in the solvability of 
\eqref{junct1}, which  
is an isometry condition between each cross-section $\{v_-=\textup{const.}\}\subset\nullhyp^-$
and its 
corresponding 
$\Phi$-image on $\nullhyp^+$.\ 
Observe that the identification of $e_1^{\pm}$ by $\Phi_{\star}$, together with \eqref{eiyl}-\eqref{eqA4}, entails that $\Phi$ maps null generators 
of $\nullhyp^-$ to null generators of $\nullhyp^+$.\  
Since $\Phi$ is a diffeomorphism, it follows that the matching requires the existence of a diffeomorphism $\Psi$ between the sets of null generators of $\nullhyp^{\pm}$.\ 

The functions 
$\{H,h^A\}$ fully encode  
the matching embedding $\phi^+$, hence all the matching information  \cite{manzano2021null}.\ 
Indeed, given coordinates $\{v_{+},u^{I}_+\}$ on $\nullhyp^+$, with $\{u_I^+\}$ constant along the generators, 
$\phi^+$ is given by
\begin{equation}
\label{embedPhi+}\phi^+(v,z^A)=\big(v_+=H(v,z^B),u^A_+=h^A(z^B)\big).
\end{equation}
The function $H$, called \textit{step (or jump) function},  
measures the jump across the null direction when crossing the matching hypersurface \cite{manzano2021null}.\ The functions $\{h^A\}$, on the other hand, rule the already mentioned identification $\Psi$ between  
the sets of null generators of $\nullhyp^{\pm}$.\ 

As mentioned before, once the identification of boundary points 
is known, the matching rigging $\rig^+$ can be computed explicitly.\ Its precise expression reads \cite[Cor. 1]{manzano2021null} 
\begin{equation}
\label{uniqxi}
\rig^+=\lp\frac{\cp H}{\cp v}\rp^{-1}\lp 
  L^+-  h_-^{AB} \frac{\cp H}{\cp z^A}
     \lp \frac{1}{2} 
   \frac{\cp H}{\cp z^B}     k^++\frac{\cp h^J}{\cp z^B}v^+_J\rp \rp.
\end{equation}
\begin{remark}
If $L^{\pm}$ are chosen null, an immediate consequence of \eqref{eqA4} and \eqref{uniqxi} is that the bases $\{\rig^+,e_a^+\}$ and $\{L^{+},k^{+},v_{I}^{+}\}$ are related by a Lorentz transformation (see e.g.\ \textup{\cite[Ch.\  2]{griffiths2009exact}}).\ 
\end{remark}
We next present the explicit expressions for the jump $[\bY]$ (cf.\ \eqref{jump:Y})
and the energy-momentum tensor of the shell.\ 
Since the expressions are rather involved in general,  
we present them only in the case when the spacetime boundaries $\nullhyp^{\pm}$ are \textit{totally geodesic}.\  
\begin{proposition}
\label{prop6}
\textup{\cite{manzano2021null}}
Assume that the matching of $(\Mpm,g^{\pm})$ across two totally geodesic null boundaries $\nullhyp^{\pm}$ is ruled by  
the functions $\{H(v,z^A),h^A(z^B)\}$.\ 
Let $h_{IJ}$ be the induced metric on the leaves $\{v=\textup{const.}\}\subset\N$,  $h^{IJ}$ its inverse tensor and $\nabh$ its Levi-Civita derivative.\
Define the vector fields $\{W_I\defi (\cp_{z^I} h^B)v_B^+\}$ on $\nullhyp^+$, and
the tensor fields
\begin{equation}
\nn \bs{\Theta}^L_{\pm}\lp X_{\pm}, Y_{\pm}\rp\defi g^{\pm}\big( \nabla_{X_{\pm}}L,{Y_{\pm}}\big), \quad \bs{\sigma}^{\pm}_L\lp X_{\pm}\rp\defi-g^{\pm}\big( \nabla_{X_{\pm}}k,L\big),\qquad\forall X_{\pm},Y_{\pm}\in \Gamma(TS^{\pm}_{v_{\pm}}).
\end{equation}
Then, the components of the tensor $[\bY]\defi \bY^{+}-\bY^-$ are given by
\begin{align}
\nn 
[\Y_{vv}]
&= - \dfrac{\cp_{ v }\cp_{ v }H}{\cp_{ v }H},\\
\label{Y1J} [\Y_{vz^J}]
&= \bs{\sigma}_{L}^+(  W_J)-\bs{\sigma}^-_{L}( v^-_J)-\dfrac{\cp_{ v }\cp_{z^J}H}{\cp_{ v }H},\\
\nn [\Y_{z^Iz^J}]
&=\frac{1}{\cp_{ v }H}\bigg(2(\nabla_{{\lp I\rd}}^hH)\bs{\sigma}_{L}^+( W_{\ld J\rp})+\bs{\Theta}^{L}_+( W_{\lp I\rd},W_{\ld J\rp})-(\cp_{ v } H)\bs{\Theta}^{L}_-( v_{\lp I \rd}^-,v_{\ld J \rp}^-)-\nabla_{I}^h\nabla_{J}^hH\bigg).
\end{align}
Moreover, the energy-momentum tensor $\tau$ of the shell reads 
\begin{align}
\nn \tau^{vv} &=-\fv h^{IJ}[\bY](\cp_{z^I},\cp_{z^J}),\\ 
\tau^{vz^I}&=\fv h^{IJ}[\bY](\cp_{ v },\cp_{z^J}),\\  
\nn \tau^{z^Iz^J}&=-\fv  h^{IJ}[\bY](\cp_{ v },\cp_{ v }),
\end{align}
where $\fv=1$ (resp.\ $\fv=-1$) if $\rig^-$ points outwards (resp.\ inwards) with respect to $\Ml$.
\end{proposition}
\begin{remark}
If the vector fields $L^{\pm}$ are chosen null, $\bs{\Theta}_{\pm}^{L}$, $\bs{\sigma}^{\pm}_L$ coincide with the transverse second fundamental form and the torsion one-form of the leaves $\{v_{\pm}=\textup{const.}\}\subset\nullhyp^{\pm}$.\
\end{remark}

\section{Matching two Minkowski spaces across a null hyperplane}\label{sec:intro:MK}

From now on, we focus our analysis on the 
matching of two regions of Minkowski space across a null hyperplane.\ In particular, in this section we shall 
prove the feasibility of the matching, provide explicit expressions for $[\bY]$ and the 
energy-momentum tensor $\tau$, and examine the explicit form of the functions $\{H, h^A\}$ depending on the type of 
shell.\ We refer to   \cite{manzano2021null,manzano2023matching} for further details.

Consider two $(\n+1)$-dimensional Minkowski spaces $(\Mkpm,\flatpm)$, with metrics
\begin{align}
	\label{Mk:spacetimes}	\flatpm=-2d\cu_{\pm} d \cv_{\pm}+\delta_{AB} dx_{\pm}^A dx_{\pm}^{B},\qquad 
	\cu_+\geq0, \quad \cu_-\leq0, 
\end{align}
and boundaries 
$\nullhyp^{\pm} \defi \{ \cu_{\pm} =0\}$.\ 
The hypersurfaces $\nullhyp_{\pm}$ are null, admit a foliation by spacelike cross-sections and are totally geodesic (because the first fundamental forms of $\nullhyp^{\pm}$ are Lie-constant along the null generators).\  
Therefore, 
the results of Section \ref{sec:matching:2021}, and in particular of Proposition \ref{prop6}, apply.\ 
We let $S^{\pm}\defi \{\cu_{\pm}=0, \cv_{\pm}=0\}$, and make the following choice for the bases $\{L^{\pm},k^{\pm},v_I^{\pm}\}$ of $\Gamma(T\Mkpm)\vert_{\nullhyp^{\pm}}$:
\begin{equation}
	\label{basis:L,k,v:Mk} 
	L^{\pm}=-\cp_{\cu_{\pm}},\qquad k^{\pm}=\cp_{\cv_{\pm}},\qquad v^{\pm}_I=\cp_{x^I_{\pm}}.
\end{equation} 
The foliation functions on $\nullhyp^{\pm}$ are then $\cv_{\pm}$, and a direct calculation \cite{manzano2021null}  
yields that the tensor fields $\bs{\sigma}_{L}^{\pm}$ and $\bs{\Theta}^{L}_{\pm}$ (cf.\ Proposition \ref{prop6}) vanish everywhere on $\nullhyp^{\pm}$.\ 

To construct metric hypersurface data that is embedded in both $\Mkp$ and $\Mkm$, 
we consider an embedding $\phi^-:\N\longhookrightarrow\nullhyp^-\subset\Mkm$ of the form
\begin{align}
	\label{phi:minus:Mk} 
	\phi^-(v, z^I) =\left(\cu_-=0,\cv_-= v, x_-^I=z^I \right ).
\end{align}
It then follows that the metric induced on the leaves $\{v=\text{const.}\}\subset\N$ is the flat metric, i.e.\ $h_{AB}=\delta_{AB}$.\ 
As mentioned in Section \ref{sec:matching:2021}, the viability of the matching relies on the solvability of 
\eqref{junct1}, which  
in the present case reads
\begin{align}
	\label{jcMkcase}
	\delta_{IJ}\vert_{\phi^-(p)} = \frac{\partial h^L}{\partial z^I}\ \frac{\partial h^K}{\partial z^J}\ \delta_{LK}\vert_{\phi^+(p)}\qquad\forall p\in \N, 
\end{align}
where the matching embedding $\phi^+:\N\longhookrightarrow\nullhyp^+\subset\Mkp$
is given by \eqref{embedPhi+}, i.e. 
\begin{align}
\label{phi:plus:Mk}
\phi^+(v, z^I) =\left(\cu_{+} =0, \cv_{+} = H(v, z^I), x_{+}^I = h^I (z^J) \right ).
\end{align}
Equation \eqref{jcMkcase}
is an isometry condition between the cross-sections 
$\phi^{\pm}(\{v=\textup{const.}\})\subset\N^{\pm}$.\  
In particular, 
it forces  $\phi^{\pm}(\{v=\textup{const.}\})$ 
to be Euclidean planes, whose isometries are translations and rotations. Now, 
given a foliation $\phi^{-}(\{v=\textup{const.}\})$ of $\nullhyp^-$, the symmetry properties of 
$(\Mkp,\flatp)$ allow one to perform the necessary combinations of rotations and translations so that 
$\{\cu_+,\cv_+,x_+^I\}$ 
verify $\cp_{z^I}x_+^J\vert_{\nullhyp^+}=\cp_{z^I}h^J=\delta_I^J$.\ 
Thus,  there always exists an isometry of $\Mkp$ which turns \eqref{jcMkcase} into a trivial equation, which guarantees its solvability.\ 
The previous reasoning enables to set $h^A=z^A$ with full generality, and 
it is then straightforward to prove that  
$$\Big\{\N,\spc\gamma=\delta_{AB}dz^A\otimes dz^B,\spc\ellc=dv,\spc \elltwo=0\Big\}$$ defines null metric hypersurface data which  
is embedded in $\Mkm$ with embedding $\phi^-$ (cf.\ \eqref{phi:minus:Mk}) and rigging $\rig^-=L^-$, as well as embedded in $\Mkp$ with embedding (recall \eqref{phi:plus:Mk}) $$\phi^+(v, z^I) =\left(\cu_{+} =0, \cv_{+} = H(v, z^I), x_{+}^I =  z^I \right )$$  and rigging  
(cf.\ \eqref{uniqxi})
\begin{equation}
	\label{uniqxi:Mk}
\rig^+=-\lp\frac{\cp H}{\cp v}\rp^{-1}\lp 
\cp_{\cu_{+}}	+  \delta^{AB} \frac{\cp H}{\cp z^A}
	\lp \frac{1}{2} 
	\frac{\cp H}{\cp z^B}     \cp_{\cv_{+}}+\cp_{x^B_{+}}\rp \rp.
\end{equation}
Observe that the combination of \eqref{eiyl}, \eqref{Mk:spacetimes}, and the choice \eqref{basis:L,k,v:Mk} of $L^-$ 
implies that $\rig^-$ points inwards 
(because $L^-$ is past), whereas $\rig^+$ points outwards (recall  
that $\cp_vH>0$).\ Consequently, the riggings $\rig^{\pm}$
satisfy the orientation condition  \hyperlink{junc-cond-FHD}{$(ii)$}, and hence the matching of $(\Mkpm,\eta^{\pm})$ across $\nullhyp^{\pm}$ is feasible.\ 
In fact, since $\nullhyp^{\pm}$ are totally geodesic,  an infinite number of matchings can be performed (see the discussion in \cite{manzano2022general,manzano2023matching}).\

The jump $[\bY]$ and the energy-momentum tensor $\tau$ are obtained by particularizing Proposition \ref{prop6} 
to the present case.\  
Specifically, 
\begin{align}
	\label{YpmMk:and:tauMk}
	[\Y_{ab}]&=-\dfrac{\cp_{z^a}\cp_{z^b}H}{\cp_{ v }H},\qquad 
	\tau^{vv}
	=\rho,  \qquad \tau^{vz^A}=j^A , \qquad \tau^{z^Az^B}=\delta^{AB}p,
\end{align}
where 
\begin{align}
	\label{def:eg:flux:pressure}  
	\rho&= -\delta^{AB}\dfrac{\cp_{z^A}\cp_{z^B}H}{\cp_{ v }H}, \qquad j^A= \delta^{AB}\dfrac{\cp_{v}\cp_{z^B}H}{\cp_{ v }H}, \qquad  p= -\dfrac{\cp_{v}\cp_{v}H}{\cp_{ v }H}  
\end{align}
are the energy density, the (components of the) energy flux, and the pressure of the shell respectively \cite{manzano2023matching,poisson2004relativist}.\ 
Observe that in the present case the matching information is entirely codified by the jump function $H$, since 
$\{h^A=z^A\}$.\ This is consistent with the fact that 
$[\bY]$ and $\tau$ depend solely on $H$.\ Notice also that no restrictions have been imposed on $H$ 
apart from $\cp_v H > 0$.\ It is hence worth discussing the types of null shells that can arise depending on the specific form of $H$ 
\cite{manzano2021null}: 
\begin{itemize}\itemsep0cm
\item[\textbf{1.}] \textbf{No-shell case.}\ This situation occurs when $[\bY]=0$, and corresponds to 
a jump function of the form  $H = a v + b_J z^J + c$, where $a, b_J, c$ are real numbers and $a > 0$.\ 
While $H$ is not completely trivial, 
there always exists 
an isometry of 
$(\Mkp,\eta^+)$ which transforms it 
into $H = v$, which is of course  consistent with the absence of a shell.\
\item[\textbf{2.}] \textbf{Purely gravitational or null-dust shell.}\
Purely gravitational shells are characterized by $[\bY]\neq0$,  $\tau=0$, while null-dust shells satisfy $\rho\neq0$, $j^A=p=0$.\ In both cases, the jump function is of the form $H = a v + \mathcal{H}(z^A)$ with $a \in\mathbb{R}_{> 0}$ and $\mathcal{H}\in\Fcal(\N)$ is a scalar function on $\N$ which is Lie-constant along the generators. The type of shell is determined by $\mathcal{H}$, which dictates whether the energy density vanishes identically or not.
\item[\textbf{3.}] \textbf{Generic shell.}\ For a \textit{generic} matching of two regions of Minkowski space across a null hyperplane, the jump function reads 
\begin{equation}
	\label{MkHpress}
	H(v,z^A)=\beta( z^A)\int\exp\lp -\int p(v,z^A) dv\rp dv+\mathcal{H}( z^A),
\end{equation}
where $p$ is the pressure of the shell, and $\beta(z^A),\mathcal{H}(z^A)$ are 
scalar functions on $\N$, Lie-constant along its generators.
The functions $\beta(z^A)$ and $p(v,z^A)$ must be compatible with the condition $\cp_vH>0$. 
\end{itemize}

\begin{figure}[t!]
\centering
\includegraphics[angle=-90,width=15cm]{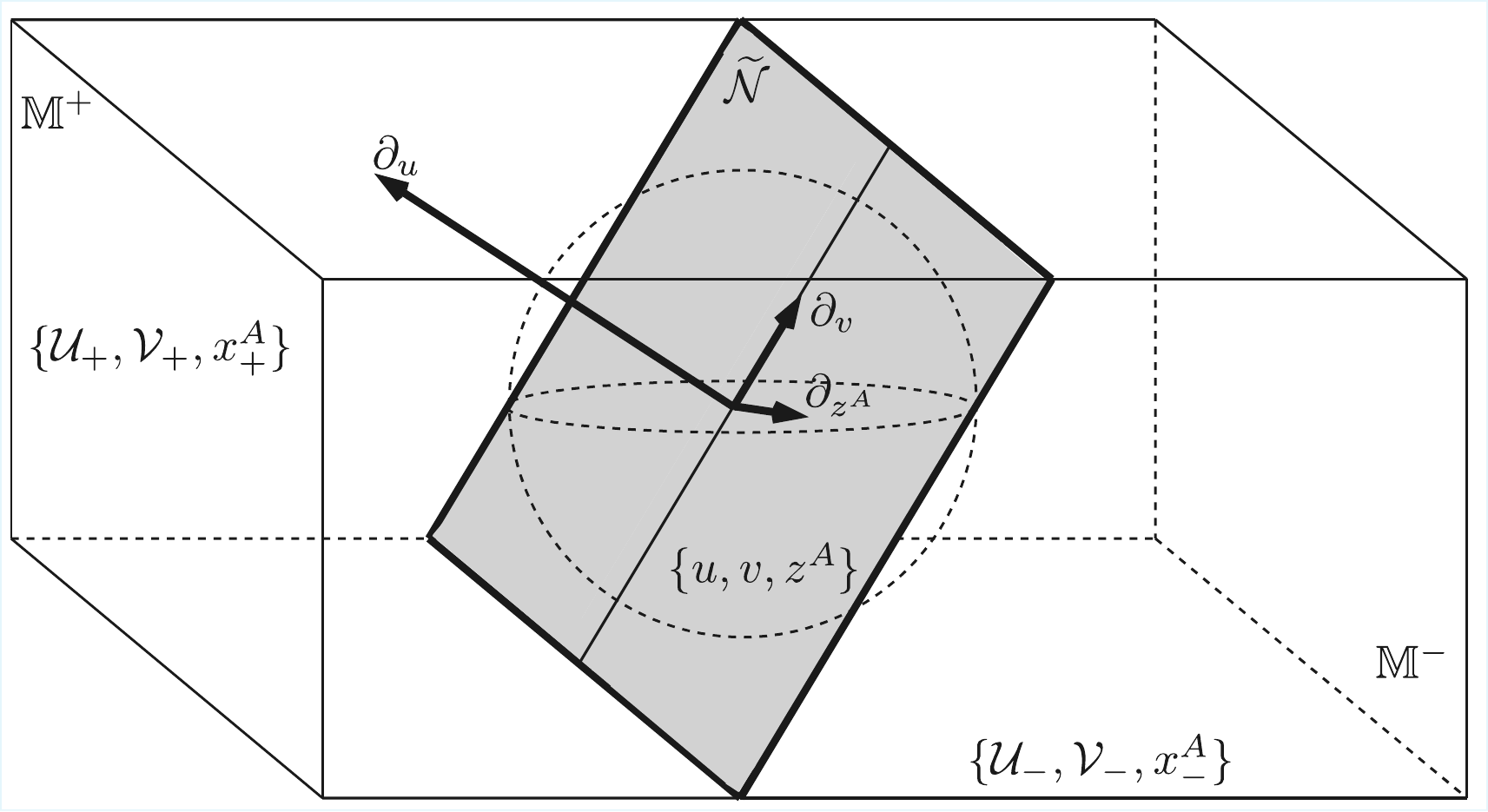}
\caption{The spacetime $(\M,g)$ arising from the most general matching of two regions of Minkowski space across a null hyperplane.\ The coordinates $\{\cu_{\pm},\cv_{\pm},x_{\pm}^A\}$ are defined on the regions $\Mkpm\cap\M$ respectively, while the coordinates $\{u,v,z^A\}$ are constructed on a (two-sided) neighbourhood of the matching hypersurface $\nullhyp\subset\M$.}
\label{Fig:3}
\end{figure}

\section{Lipschitz continuous metric form in the Minkowski case}\label{sec:lip:metric}

In view of the results of Section \ref{sec:intro:MK}, the following natural questions arise 
concerning the metric $g$ of the spacetime $\M$ constructed there from the most general matching of two Minkowski spaces $(\Mkpm,\flatpm)$ across a null hyperplane:
\begin{multicols}{2}
\begin{itemize}\itemsep0cm
    \item[(1)] What is the regularity of $g$? 
    \item[(2)] Can we find an explicit form of $g$?
\end{itemize}
\end{multicols}
The general expectation concerning (1) is that $g$ should be at least continuous (as predicted by a result in \cite{clarke1987junction}), but we prove that it is in fact locally Lipschitz continuous\footnote{Observe that, in the last couple of years, it has become clear that local Lipschitz continuity is the threshold regularity necessary for a proper Lorentzian causality theory, cf.\ \cite{CG:12,GKSS:20}.}. Indeed, in Appendix \ref{app:B} we even show a corresponding result for the matching of general semi-Riemannian manifolds across an arbitrary hypersurface. 

In this section we will address (2) 
by first constructing a coordinate system in a neighbourhood of the matching hyperplane and then deriving an explicit Lipschitz continuous expression for $g$ which also encodes 
the physical properties of the shell in a transparent way.

Let $(\M,g)$ be the spacetime resulting from the matching of $(\Mkpm, \flatpm)$ across $\nullhyp^{\pm}$ with jump function \eqref{MkHpress}, and $\nullhyp\subset\M$ the corresponding matching hypersurface (see Figure \ref{Fig:3}).\ Then $\M\equiv(\Mkp\cup \Mkm)/ \nullhyp$ (or $\M\equiv\Mkp \cup_{\Phi} \Mkm$ in the notation of Appendix \ref{app:B}), where the quotient means that we are identifying the boundaries $\nullhyp^{\pm}$ and this identification gives rise to a single null hypersurface ($\nullhyp$) where 
the null shell is located. In the coordinate systems  $\{\cu_{\pm},\cv_{\pm},x_{\pm}^A\}$ for $\lp\Mkpm,\flatpm\rp$ used above,  $\nullhyp$ is given by $\nullhyp\equiv\{\cu_{\pm}=0\}$. 

Our aim now is to build a \textit{new} coordinate system $\{u,\hatv,\hatz^A\}$
in a neighbourhood $\mathcal{O}\subset\mathcal{M}$ of $\nullhyp$ and to explicitly derive metric there. We start by enforcing the following trivial identification between the coordinate systems 
$\{u,\hatv,\hatz^A\}$ and $\{\cu_{-},\cv_{-},x_{-}^A\}$ on $\Mkm$:
\begin{equation}
\label{minus:coords:uvz:coords}
\ld \Bigl\{\cu_-=u,\spc\cv_-=\hatv,\spc x_-^A=\hatz^A\Bigr\}\rv_{\Mkm}. 
\end{equation}
Then \eqref{eiyl} together with the choice \eqref{basis:L,k,v:Mk} force the vector fields $\{\rig^-,e_a^-\}$ to be given by 
\begin{equation}
\label{minussideVF}
e_1^-=\cp_{\hatv},\qquad e_I^-=\cp_{\hatz^I},\qquad\rig^-=-\cp_{u},
\end{equation} 
in the basis $\{\cp_{u},\cp_{\hatv},\cp_{\hatz^A}\}$ 
of $\Gamma\big( T\mathcal{M}\big)\vert_{\nullhyp}$.\ 
The bases $\{e_a^{\pm},\rig^{\pm}\}$ are identified in the matching process, so the rigging $\rig^+$ and the vector fields $\{e^+_a\}$ must verify (recall \eqref{eqA4},  \eqref{uniqxi:Mk} and the choice $h^A(z^B)=z^A$)
\begin{align}
	\label{plussideVF1a}
	e_1^+&=(\cp_{ v }H) \cp_{\cv_+}=\cp_{\hatv},\quad\qquad e_I^+=(\cp_{z^I}H) \cp_{\cv_+}+\cp_{x_+^I}=\cp_{\hatz^I},\\
	\label{plussideVF1b}\rig^+&= -\frac{1}{\partial_{ v } H} \left (
	\cp_{\cu_+} +  \delta^{AB} (\partial_{z^A} H)  \left ( \frac{1}{2}  (\partial_{z^B} H )  \cp_{\cv_+}
	+ \cp_{x_+^B} \right ) \right )=-\cp_{u}.
\end{align}
Since only tangential derivatives of  
$H$ appear on \eqref{plussideVF1a}-\eqref{plussideVF1b}, and 
\begin{equation}
\nn \Bigl\{ v=\cv_-=\hatv,\spc z^A=x_-^A=\hatz^A\Bigr\}\text{ on $\N$ (cf.\ \eqref{phi:minus:Mk}),}
\end{equation}
from now we drop the hat from the coordinates 
$\{\hatv,\hatz^A\}$
and simply write $\{u,v,z^A\}$.\ This, in particular, enables to rewrite  
\eqref{plussideVF1a}-\eqref{plussideVF1b} as
\begin{align}
\label{plussideVF2a}
e_1^+&=(\cp_{v}H) \cp_{\cv_+}=\cp_v,\quad\qquad e_I^+=(\cp_{z^I}H) \cp_{\cv_+}+\cp_{x_+^I}=\cp_{z^I},\\
\label{plussideVF2b}\rig^+&=  -\frac{1}{\partial_{v} H} \left (
\cp_{\cu_+} +  \delta^{AB} (\partial_{z^A} H)  \left ( \frac{1}{2}  (\partial_{z^B} H )  \cp_{\cv_+}
+ \cp_{x_+^B} \right ) \right )=-\cp_{u}.
\end{align}
From \eqref{plussideVF2a}-\eqref{plussideVF2b}, we shall  
obtain an explicit relation between the coordinates $\{\cu_+,\cv_+,x_+^A\}$ and $\{u,v,z^A\}$ as follows.\  
Let $y_+$ denote any of $\cu_+$, $\cv_+$ or $x_+^A$, and $\nabla^+$ be the Levi-Civita covariant derivative of $\flatp$.\ First, we notice that 
the one-forms $d\cu_+$, $d\cv_+$ and $dx^A_+$ are covariantly constant on $\Mkp\subset\M$, since   
\begin{equation}
\nabla^+_{\alpha}(dy_+)_{\beta}=\nabla^+_{\alpha}\nabla^+_{\beta}y_+=\cp_{\alpha}\cp_{\beta}y_+=0.
\end{equation}
Secondly, we enforce that the vector field $\xi\defi \cp_u$ is null and affinely geodesic everywhere on $\Mkp$.\ This implies 
\begin{align}
\nn &\xi\big(\xi(y_+)\big)= \xi^{\alpha}\nabla^+_{\alpha}\big(  \xi^{\beta}\nabla^+_{\beta}y_+\big)=\big( \xi^{\alpha}\nabla^+_{\alpha} \xi^{\beta}\big)\big( \nabla^+_{\beta}y_+\big)+ \xi^{\alpha} \xi^{\beta}\nabla^+_{\alpha} \nabla^+_{\beta}y_+=0\\
\label{second:Der:coords}&\Longleftrightarrow\quad y_+=a+ub,
\end{align}
where $a\lp v,z^A\rp$ and $b\lp v,z^A\rp$ are scalar functions on 
$\Mkp$ with no dependence on $u$.\ The coordinate transformation on the region $\Mkp$ of $\M$ must therefore be of the form 
\begin{equation}
\label{coord:transf:general}\cu_+=\cu_0+u\hspace{0.075cm}\cu_1,\qquad\cv_+=\cv_0+u\cv_1,\qquad x^A_+=x^A_0+u x^A_1,
\end{equation}
where $\cu_0,\cu_1,\cv_0,\cv_1,x^A_0,x^A_1$ are scalar functions on $\Mkp$ that only depend on $\{v,z^A\}$.\ 
Finally, we use the relations \eqref{phi:plus:Mk} and \eqref{plussideVF2b} to determine $\cu_0,\cu_1,\cv_0,\cv_1,x^A_0,x^A_1$ explicitly.\ This requires that we extend scalar functions on $\N$ to functions on $\Mkp\subset\M$, which we accomplish by requiring that they are independent of $u$.\ 
Specifically, the trivial identification of coordinates on $\Mkm\subset\M$ entails that $\nullhyp\equiv\{u=0\}$.\ This, together with $h^A=z^A$ and \eqref{phi:plus:Mk}, forces $\cu_0=0$, $\cv_0=H\lp v,z^A\rp$, $x^A_0=z^A$.\ Therefore, 
\begin{align}
\label{eq1a:0}&\cu_+=u\hspace{0.075cm}\cu_1,&&\cv_+=H+u\cv_1,&& x^A_+=z^A+u x^A_1,&\\
\label{eq1a} &d\cu_+=\cu_1du+ud\cu_1,&&d\cv_+=dH+\cv_1du+ud\cv_1,&& dx^A_+=dz^A+ x^A_1du+u dx^A_1.&
\end{align}
The scalar functions $\cu_1$, $\cv_1$, $x^A_1$ can be derived from \eqref{plussideVF2b} by decomposing $\cp_u$ as
\begin{equation}
\label{eq:cp_u:UVX}\cp_u=\frac{\cp \cu_+}{\cp u}\cp_{\cu_+}+\frac{\cp \cv_+}{\cp u}\cp_{\cv_+}+\frac{\cp x^A_+}{\cp u}\cp_{x^A_+}=\cu_1\cp_{\cu_+}+\cv_1\cp_{\cv_+}+x_1^A\cp_{x^A_+}.
\end{equation}
Indeed, inserting \eqref{eq:cp_u:UVX} into \eqref{plussideVF2b}, one obtains (recall that $\cp_vH>0$)
\begin{align}
\label{u1v1x1}&\cu_1=\dfrac{1}{\cp_vH},\qquad\cv_1=M\cu_1,\qquad x^A_1=\qfunction^A\cu_1\\
\nn &\textup{where}\qquad M\defi \dfrac{1}{2}\delta^{AB}\cp_{z^A}H\cp_{z^B}H,\quad\qfunction^A\defi \delta^{AB}\cp_{z^B}H.
\end{align}
Observe that the quantities $\qfunction^A$ and $\cu_1$ verify
\begin{align}
\label{id:for:dqA:and:dU1:a} d\qfunction^A&=\delta^{AB}\lp \cp_{v}\cp_{z^B}Hdv+\cp_{z^C}\cp_{z^B}Hdz^C\rp,\\
\label{id:for:dqA:and:dU1:b} d\cu_1&=-\cu_1^2\lp \cp_{v}\cp_{v}Hdv+\cp_{z^B}\cp_{v}Hdz^B\rp.
\end{align}
Let us now derive the explicit expressions of the flat metrics $\flatpm$ in the coordinates $\{u,v,z^A\}$.\ The calculation for $\flatm$ is immediate because of the trivial relation between $\{u,v,z^A\}$ and $\{\cu_-,\cv_-,x^A_-\}$ on $\Mkm$ (cf.\ \eqref{minus:coords:uvz:coords}), which in particular entails
\begin{align}
\label{one-forms:minus:side}d\cu_-\stackbin{\Mkm}= du,\quad d\cv_-\stackbin{\Mkm}=dv,\quad dx_-^A\stackbin{\Mkm}=dz^A,
\end{align} 
and hence 
\begin{align}
\label{flatm:uvz}\flatm=-2dudv+\delta_{AB}dz^Adz^B.
\end{align}
On the other hand, the results \eqref{eq1a:0}-\eqref{eq1a} and \eqref{u1v1x1} allow us to rewrite $\flatp$ as follows:\ 
\begin{align}
\nonumber \flatp=&-2d\cu_+d\cv_++\delta_{AB}dx_+^Adx_+^B\\
\nonumber =&-2\lp \cu_1du+ud\cu_1\rp \lp dH+\cv_1du+ud\cv_1\rp\\
\nonumber &+\delta_{AB}\lp dz^A+ x^A_1du+u dx^A_1\rp \lp dz^B+ x^B_1du+u dx^B_1\rp\\
\nonumber =&\spc du^2\lp -2\cu_1\cv_1+\delta_{AB}x^A_1x^B_1\rp\\
\nonumber &+2du\Big( u\lp -\cv_1d\cu_1-\cu_1d\cv_1+\delta_{AB}x_1^Bdx_1^A\rp-\cu_1dH+\delta_{AB}x_1^Adz^B\Big)\\
\label{mas:Eqs:yoyanosequeponer} &-2ud\cu_1dH-2u^2d\cu_1d\cv_1+\delta_{AB}\lp dz^A+udx_1^A\rp \lp  dz^B+udx_1^B\rp.
\end{align}
Now, the vector field $\cp_u$ being null everywhere forces (this also follows directly from \eqref{u1v1x1}) 
\begin{equation}
\label{otraeq.whatever.name}-2\cu_1\cv_1+\delta_{AB}x^A_1x^B_1=0\qquad\Longrightarrow\qquad0=-\cv_1d\cu_1-\cu_1d\cv_1+\delta_{AB}x^A_1dx^B_1.
\end{equation}
Besides, since our choice of rigging $\rig^-$ is orthogonal to the spacelike sections $\{\cv_{-}=\text{const.}\}$ of $\nullhyp^-$, 
we know by \eqref{minus:coords:uvz:coords}-\eqref{minussideVF} that $\cp_u$ is orthogonal to the spacelike sections of $\{v=\text{const.}\}$ of $\N$ as well.\ In particular, this implies that (again this can also be obtained from \eqref{u1v1x1} via a straightforward calculation)
\begin{equation}
\label{ffifsjodiajH432RY8}-\cu_1\cp_{z^A}H+\delta_{AB}x^B_1=0\qquad\Longleftrightarrow\qquad-\cp_{z^A}H+\delta_{AB}\delta^{BJ}\cp_{z^J}H=0.
\end{equation}
Since $dH=\cp_vHdv+\cp_{z^A}Hdz^A$, it follows that 
\begin{align}
\nn -\cu_1 dH+\delta_{AB}x_1^Adz^B&= \cu_1 (- dH+\delta_{AB}\qfunction^Adz^B)\\
\label{arriba} &= \cu_1(- \cp_vHdv-\cp_{z^A}Hdz^A+\delta_{AB}\delta^{BJ}\cp_{z^J}Hdz^A)=-dv,
\end{align}
where we have used \eqref{ffifsjodiajH432RY8} to cancel the last to terms and \eqref{u1v1x1} to introduce the coefficient $\cu_1$. Inserting \eqref{otraeq.whatever.name} and \eqref{arriba} into \eqref{mas:Eqs:yoyanosequeponer} yields
\begin{align}
\nonumber \flatp=&-2dudv-2ud\cu_1dH-2u^2d\cu_1d\cv_1+\delta_{AB}\lp dz^A+udx_1^A\rp \lp  dz^B+udx_1^B\rp\\
\nonumber =&-2dudv +\delta_{AB}dz^Adz^B-2u\lp d\cu_1dH-\delta_{AB}dz^Adx_1^B\rp
-u^2\lp 2d\cu_1d\cv_1-\delta_{AB}dx_1^Adx_1^B\rp\\
\nonumber =&-2dudv +\delta_{AB}dz^Adz^B-2u\Big( d\cu_1dH-\delta_{AB}dz^A\lp \cu_1d\qfunction^B+\qfunction^Bd\cu_1\rp\Big)\\
\nonumber &-u^2\Big( 2d\cu_1\lp Md\cu_1+\cu_1dM\rp-\delta_{AB}\lp \cu_1d\qfunction^A+\qfunction^Ad\cu_1\rp\lp \cu_1d\qfunction^B+\qfunction^Bd\cu_1\rp\Big),
\end{align}
which upon using \eqref{arriba} as well as the identities  $2M-\delta_{AB}\qfunction^A\qfunction^B=0$, $dM-\delta_{AB}\qfunction^Ad\qfunction^B=0$ that follow from \eqref{u1v1x1}, gives 
\begin{align}
\nonumber \flatp=&-2dudv +\delta_{AB}dz^Adz^B+2u\lp -\frac{1}{\cu_1}dvd\cu_1+\cu_1\delta_{AB}dz^A d\qfunction^B\rp\\
\nonumber &+u^2\lp d\cu_1\lp -2\cu_1dM+2\cu_1\delta_{AB}\qfunction^Ad\qfunction^B\rp+\cu_1^2\delta_{AB}d\qfunction^Ad\qfunction^B\rp\\
\nonumber =&-2dudv +\delta_{AB}dz^Adz^B+2u\lp -\frac{1}{\cu_1}dvd\cu_1+\cu_1\delta_{AB}dz^A d\qfunction^B\rp+u^2\cu_1^2\delta_{AB}d\qfunction^Ad\qfunction^B.
\end{align}
Finally, inserting \eqref{id:for:dqA:and:dU1:a}-\eqref{id:for:dqA:and:dU1:b} and using \eqref{YpmMk:and:tauMk} and the expression in \eqref{def:eg:flux:pressure} for the pressure, it is straightforward to check that in the coordinates $\{u,v,z^A\}$ the metric $\flatp$ reads 
\begin{align}
\nonumber \flatp=&-2dudv +\delta_{AB}dz^Adz^B+udv^2\lp u\delta^{AB}[\Y_{vz^A}][\Y_{vz^B}]-2p\rp\\
\nn  &+2u[\Y_{vz^I}]dvdz^A\lp -2\delta^I_A+u\delta^{BI}[\Y_{z^Az^B}]\rp\\
\label{eq5} &-2udz^Adz^B\lp [\Y_{z^Az^B}]-\frac{u}{2}\delta^{IJ}[\Y_{z^Iz^A}][\Y_{z^Jz^B}]\rp.
\end{align}
The coordinates $\{u,v,z^A\}$ are no longer Minkowskian on $\Mkp$.\ Instead, $\flatp$ contains second derivatives of the jump function $H(v,z^A)$ (cf.\  \eqref{YpmMk:and:tauMk}-\eqref{def:eg:flux:pressure}).\ Note that, in the no-shell case (see Section \ref{sec:intro:MK}) where the jump function is linear, $\{u,v,z^A\}$ become Minkowskian and one recovers the global Minkowski space.\ This represents a non-trivial consistency check for \eqref{eq5}.\

From \eqref{flatm:uvz} and \eqref{eq5}, we now see that $\flatm$ and $\flatp$ agree on $\nullhyp=\{u=0\}$ and can be written in the following explicit Lipschitz continuous form in terms of $\{u,v,z^A\}$ 
as 
\begin{align}
\nn g=&-2dudv +\delta_{AB}dz^Adz^B+u_+( u)dv^2\lp u\delta^{AB}[\Y_{vz^A}][\Y_{vz^B}]-2p\rp\\
\nn  &+2u_+( u)[\Y_{vz^I}]dvdz^A\lp u\delta^{BI}[\Y_{z^Az^B}]-2\delta^I_A\rp\\
\label{c0metricprevious:}&-2u_+( u)dz^Adz^B\lp [\Y_{z^Az^B}]-\frac{u}{2}\delta^{IJ}[\Y_{z^Iz^A}][\Y_{z^Jz^B}]\rp,
\end{align}
where $$u_+\lp u\rp\defi\lb\begin{array}{ll}u & u\geq0\\ 0 & u<0\end{array}\rd$$ is the so-called kink function.\ We can summarize the previous results as follows. 

\begin{theorem}\label{theorem:C0metric:Mk} 
Let $\lp\mathcal{M},g\rp$ be the spacetime resulting from the most general matching of two regions $\lp\Mkpm,\flatpm\rp$ of Minkowski space across a null hyperplane (cf.\ \eqref{MkHpress}).\ Denote by $\nullhyp$ the matching hypersurface.\ 
Then, there exists coordinates $\{u,v,z^A\}$ on a neighbourhood $\mathcal{O}\subset\M$ of $\nullhyp$ with the following properties:
\begin{itemize}\itemsep0cm
	\item[(i)] $\{u,v,z^A\}$ are Gaussian null coordinates\footnote{See e.g.\ \cite{kunduri2013classification, moncrief1983symmetries} 
	for details on the construction of Gaussian null coordinates.} 
	on both sides of $\nullhyp$, such that 	$\cp_{v}\vert_{\nullhyp}$ is a 
	null generator of $\nullhyp$, and 	$\partial_u\vert_{\nullhyp}$
	is a rigging of $\nullhyp$ with the properties of being future-directed, null, orthogonal to 
		the spacelike planes $\nullhyp \cap \{v = \textup{const.}\}$ and satisfying
		$g (\partial_u, \partial_{v} )\vert_{\nullhyp} = -1$.
		\item[(ii)] In the coordinates $\{u,v,z^A\}$, the metric $g\vert_{\mathcal{O}}$ takes the Lipschitz continuous form 
		\begin{align}
			\nonumber g\vert_{\mathcal{O}}=&-2dudv +\delta_{AB}dz^Adz^B+u_+( u)dv^2\lp u\delta^{AB}[\Y_{vz^A}][\Y_{vz^B}]-2p\rp\\
			\nonumber  &+2u_+( u)[\Y_{vz^I}]dvdz^A\lp u\delta^{BI}[\Y_{z^Az^B}]-2\delta^I_A\rp\\
			\label{c0metricprevious}&-2u_+( u)dz^Adz^B\lp [\Y_{z^Az^B}]-\frac{u}{2}\delta^{IJ}[\Y_{z^Iz^A}][\Y_{z^Jz^B}]\rp,
		\end{align}
		where $u_+\lp u\rp$ is the kink function.
\item[$(iii)$]
	In the case of a purely gravitational or a null-dust shell, where 
	the jump function is given by $H=a v+\mathcal{H}(z^A)$, $a\in\mathbb{R}_{>0}$, $\mathcal{H}(z^A)\in\Fcal(\mathcal{O})$,   
	$g\vert_{\mathcal{O}}$ becomes (cf.\ \eqref{YpmMk:and:tauMk}) 
	\begin{align}
		\nn g=&-2dudv +\delta_{AB}dz^Adz^B\\
		\label{c0metricGW}&+\frac{2u_+( u)}{a}\lp \cp_{z^A}\cp_{z^B}\mathcal{H}
		+\frac{u\delta^{IJ}}{2a}(\cp_{z^I}\cp_{z^A}\mathcal{H})(\cp_{z^J}\cp_{z^B}\mathcal{H})\rp dz^Adz^B.
	\end{align}
	\end{itemize}
\end{theorem}
Observe that the metric $g$ by the above is locally Lipschitz continuous but it is even piecewise smooth, as it is smooth in the regions $\Mkpm\setminus\nullhyp\subset\M$. 
\begin{remark}\label{rem:specialcase}
	In the purely gravitational or null-dust case, if the spacetimes 
	$(\Mkpm,\flatpm)$ are $4$-dimensional, one can define complex spatial  coordinates 
	\begin{equation}
		\mathcal{Z}\defi \frac{z^2+iz^3}{\sqrt{2}},\qquad \ov{\mathcal{Z}}\defi \frac{z^2-iz^3}{\sqrt{2}},
	\end{equation}
	where $i$ denotes the imaginary unit, and then a direct calculation shows that the one-form
	\begin{equation}
		\label{def:omega:form}\bs{\omega}\defi d\mathcal{Z}+\frac{u_+(u)}{a}\Big( (\cp_{\ov{\mathcal{Z}}}\cp_{\mathcal{Z}}\mathcal{H}) d\mathcal{Z}+(\cp_{\ov{\mathcal{Z}}}\cp_{\ov{\mathcal{Z}}}\mathcal{H}) d\ov{\mathcal{Z}}\Big)
	\end{equation}
	satisfies (we let $\ov{\bs{\omega}}$ be the conjugate of $\bs{\omega}$)
	\begin{equation}
		\nn 2\bs{\omega}\otimes_s\ov{\bs{\omega}}=(dz^2)^2+(dz^3)^2+\frac{2u_+(u)}{a}(\cp_{z^J}\cp_{z^B}\mathcal{H}) \lp \delta_A^J+\frac{u}{2a}\delta^{IJ}\cp_{z^I}\cp_{z^A}\mathcal{H}\rp dz^Adz^B.
	\end{equation}
	Thus, the metric \eqref{c0metricGW} can be expressed in terms of $\bs{\omega}$ as
	\begin{align}
		\label{c0metricGW:4dim} g=-2dudv +2\bs{\omega}\otimes_s\ov{\bs{\omega}}, 
	\end{align}
	which is the standard way of writing the Lipschitz continuous metric $g$ when studying the matching of two 
	Minwkoski spaces across a null hyperplane by means of the cut-and-paste method (see e.g.\ \textup{\cite{penrose1968twistor,podolsky2017penrose}}).\ 
	As discussed in \textup{\cite{manzano2021null}}, the positive real number $a$ can always be set to one without loss of generality 
	by performing a boost in $(\Mkp,\flatp)$.\ In that case,  
	\eqref{c0metricGW:4dim} becomes \textup{\cite[Eq.\ (5)]{podolsky2017penrose}}.\   
	\end{remark}

\section{Distributional metric form in the Minkowski case}\label{sec:dist:metric}

In this section we derive an alternative distributional metric for the matched spacetime $(\M,g)$ constructed in Section \ref{sec:intro:MK} with metric \eqref{c0metricprevious}. To this end we will use Penrose's \textit{cut-and-paste} method \cite{Penrose:1972xrn} as a strong guidance, cf.\ Section \ref{sec:intro}.\ We start by recalling the essence of this method, in the present context: First cut Minkowski space along the null hyperplane $\nullhyp=\{{\mathcal U}_\pm=0\}$, splitting it into two separate regions $(\Mkpm,\flatpm)$ as defined in \eqref{Mk:spacetimes}, and then reattach the two ``halves" $\Mkpm$ with a warp along the null generators of the matching hypersurface $\nullhyp$
given by 
\begin{equation}
\label{penrose:junct:cond1}\bigl\{\cu_+=\cu_-=0,\quad\cv_+=\cv_-+\mathcal{H}(x_-^A),\quad x_+^A=x_-^A\bigr\}
\end{equation}
in the flat coordinates introduced in \eqref{Mk:spacetimes}.\ As a consequence of \eqref{phi:minus:Mk}-\eqref{phi:plus:Mk}, conditions \eqref{penrose:junct:cond1} amount to impose a jump function of the form $H=v+\mathcal{H}(z^A)$ which, as discussed in Section \ref{sec:intro:MK}, gives rise to purely gravitational or null-dust shells.\ Observe that this is a first strong connection between the methods employed above an the cut-and-paste approach, see \cite[Sec.\ 6]{manzano2021null}.\  The second such connection was established in Remark \ref{rem:specialcase}. Indeed, Penrose's impulsive plane waves in Rosen form \cite[p.\ 103]{Penrose:1972xrn} (see also \cite{AB:97,PV:98}) are precisely a special case of \eqref{c0metricGW:4dim}. 

As mentioned before, the cut-and-paste method allows to write 
the metric of the resulting spacetime in terms of a Dirac delta distribution $\delta$, cf.\ \eqref{CAP:Mk:metric}, where the coordinates $\{\cu,\cv,\cx,\mathcal{Y}\}$ coincide with $\{\cu_{\pm},\cv_{\pm},x_{\pm}^2,x_{\pm}^3\}$ in the regions $\Mkpm\setminus\nullhyp$.\ The key advantages of the distributional form \eqref{CAP:Mk:metric} are that the spacetime metric of the matched spacetime takes the same, manifestly flat form in the spacetime regions before and after the shell, and that the shell’s information can be entirely encoded in the impulsive term of the metric.\ This, however, comes at the price that, at face value, the metric \eqref{CAP:Mk:metric} has to be considered as being only formal as it firmly lies beyond the threshold regularity for metrics that can reasonably be assigned a distributional curvature, which is locally $H^1\cap L^\infty$ \cite{geroch1987strings,lefloch2007definition,SV:09}. In the particular case of \eqref{CAP:Mk:metric}, the high degree of symmetry allows to (formally) calculate the curvature, since no ill-defined products of distributions arise.\ Moreover, such low regularity spacetimes can in general be handled rigorously using nonlinear distributional geometry \cite[Ch.\ 3]{grosser2013geometric}, which is especially true for the notorious ``discontinuous coordinate transformation'' which Penrose used to (formally) demonstrate the equivalence of both his forms of the metric, cf.\ \cite{KS:99,samann2023cut}.

In this section, we will follow this circle of ideas to generalize the Penrose distributional metric and his ``discontinuous coordinate transformation'' in a formal sense. In particular, the generalization we aim at is the following. In all cut-and-paste constructions so far, the warp \eqref{penrose:junct:cond1} is such that any two points along a null generator experience the \textit{same} jump at the matching hypersurface.\ This restriction, while still allowing for some generality due to the freedom in the choice of $\mathcal{H}$, significantly limits the class of shells under consideration, and in particular excludes those with nonzero energy flux and pressure (recall \eqref{YpmMk:and:tauMk}). This raises the quest to extend the cut-and-paste construction by Penrose \textit{et al}.\ to generic matchings of two regions of Minkowski across a null hyperplane. More specifically, it is worth investigating whether the metric 
\eqref{c0metricprevious} of the spacetime resulting from a generic matching of $(\Mkpm,\flatpm)$ across the hyperplanes $\nullhyp^{\pm}$ can be written in a distributional form in such a way that the full information of the shell is encoded in an impulsive Dirac-delta term.\ This is indeed feasible, as we shall see next. 

Let $\lp\mathcal{M},g\rp$ be the spacetime of Theorem \ref{theorem:C0metric:Mk}. Then, based on the coordinates $\{u,v,z^A\}$ of Theorem \ref{theorem:C0metric:Mk}(i) we define a coordinate system $\{\cu,\cv,\cx^A\}$  via the ``discontinuous coordinate transform''
\begin{align}
	\label{disc:coord:transf:u} \cu&\defi u+\up\lp \cu_1-1\rp,\\ 
	\label{disc:coord:transf:v}\cv&\defi v+\heav\lp H-v\rp+\up\cv_1,\\ 
	\label{disc:coord:transf:x} \cx^A&\defi z^A+\up x_1^A.
\end{align}
Now we claim that the metric \eqref{c0metricprevious} of Theorem \ref{theorem:C0metric:Mk}(ii) transforms to
\begin{equation}
	\label{metric:distributional:Mk}
	g=-2d\cu d\cv+\delta_{AB}d\cx^A \cx^B+2\tcte{\hor(v,z^C)\delta(u)} d\cu^2,
\end{equation}
where $\delta$ is the Dirac function and $\hor$ is a scalar function defined by 
\begin{equation}
\label{def:hor}\hor(v,z^C)\defi (\cp_vH)\frac{1+(\cp_vH)}{1+(\cp_vH)^2}\lp H-v\rp. 
\end{equation}
To establish this claim, we carefully compute the corresponding terms keeping precise track of the regularities used especially in the nonlinear terms.  

We start with the calculation of $d\cu$ and $d\cx^A$.\ Since 
$\{\cu,\cx^A\}$ are Lipschitz continuous, they are
differentiable almost everywhere by Rademacher's theorem.  
Therefore, to calculate $\{d\cu,d\cx^A\}$ we perform the ``piecewise almost everywhere'' calculation\footnote{Here we write $u_+=\theta(u)u$ as a product of $L^\infty$-functions.}
\begin{align}
\label{dcxA} d\cx^A=&\spc 
dz^A + \theta(u)\lp x_1^A du+udx^A_1 \rp 
= \big( 1-\theta(u)\big) dx_-^A + \theta(u) dx_+^A,\\
\label{dcu} d\cu=&\spc 
\big(1-\theta(u)\big) du +\theta(u)\big(\cu_1 du+ud\cu_1\big)=\big(1-\theta(u)\big)d\cu_-+\theta(u)d\cu_+,
\end{align}
where in the last steps we used \eqref{eq1a} and \eqref{one-forms:minus:side}.\ 

Now the coordinate $\cv$ is only $L^\infty$. So, in order to compute the one-form $d\cv$ we need to consider it in the space of distributions, i.e. we use $L^{\infty}_{\textup{loc}}\subseteq L^{1}_{\textup{loc}} \subseteq\mathcal{D}'$ and the product rule for the product $C^\infty\cdot\mathcal{D}'$. We obtain 
\begin{align}
\nn d\cv=&\spc dv+\theta(u) \cv_1 du +u_+(u)d\cv_1+\theta(u)(dH-dv)+\delta(u)(H-v)du\\
=&\spc \big(1-\theta(u)\big) dv+\theta(u)\big( \cv_1 du +ud\cv_1+dH\big)+\delta(u)(H-v)du\\
\label{dcv}=&\spc \big(1-\theta(u)\big) d\cv_-+\theta(u)d\cv_++\delta(u)(H-v)du,
\end{align} 
where in the last step we again used \eqref{eq1a} and \eqref{one-forms:minus:side}.

The spatial part of the metric \eqref{metric:distributional:Mk}, $\delta_{AB}d\cx^A d\cx^B$, again is Lipschitz and so another ``almost everywhere calculation'' leads to
\begin{align}
\label{dist:metric:term:1} \delta_{AB}d\cx^A d\cx^B&=\lb
\begin{array}{l}
\delta_{AB}dx_-^A dx_-^B\vert_{u<0}\\  
\delta_{AB} dx_+^A dx_+^B\vert_{u>0}  \\
\end{array}
\rd
=\delta_{AB}\Big( \big(1-\theta(u)\big)dx_-^Adx_-^B+\theta(u)dx_+^Adx_+^B\Big).
\end{align}

By following an analogous reasoning, we can compute the product $d\cu^2$ to
\begin{align}
\label{dcu:squared}d\cu^2&=\lb
\begin{array}{l}
d\cu_-^2\vert_{u<0}\\  
d\cu_+^2\vert_{u>0}  \\
\end{array}
\rd
= \big(1-\theta(u)\big)d\cu_-^2+\theta(u)d\cu_+^2.
\end{align}

In the term $-2d\cu d\cv$, for the first time, we encounter the ``problematic product'' $L^{\infty}_{\textup{loc}}\cdot \mathcal{D}'$ which is undefined in standard distribution theory and we resort to regularisation. More precisely, we take the following approach which is a sensible modelling of the situation at hand. Viewing the shell as a limiting case of thick shells with an arbitrarily regularized profile we regularize the Dirac-$\delta$ in a very general way, i.e.\ by a so called model delta net $\rho_\eps$, see Appendix \ref{app:distributional:identities}. But then we naturally view the Heaviside function $\theta$ as originating from the same thick shell and hence regularize it by the prime function of $\rho_\eps$, i.e.\ of the model delta net that represents the Dirac-$\delta$. Hence we regularize the $\theta$-function by $\theta_\eps(x)=\int^x\rho_\eps(s)ds$. Then we calculate the regularisation product of $\theta\,\delta$ (the model product in the language of  \cite[Ch.\ II § 7]{Obe:92}), i.e. the weak limit of $\rho_\eps\,\theta_\eps$, to obtain $\frac{1}{2}\delta$, see Appendix \ref{app:distributional:identities} for details. In the same vein we obtain the regularisation product $\theta^2=\theta$ and the following consequences which we will use below
\begin{equation}
\label{products:2}\big(1-\theta(u)\big)\cdot \big(1-\theta(u)\big)=1-\theta(u),\quad\big(1-\theta(u)\big)\cdot\theta(u)=0,\quad\big(1-\theta(u)\big)\cdot\delta(u)=\frac{1}{2}\delta(u).
\end{equation}

With this regularisation products at hand we obtain (recall \eqref{dcu}-\eqref{dcv})
\begin{align}
\nn -2d\cu d\cv=&-2 \Big(\big(1-\theta(u)\big)d\cu_-+\theta(u)d\cu_+\Big)  \Big(\big(1-\theta(u)\big) d\cv_-+\theta(u)d\cv_++\delta(u)(H-v)du\Big)\\
\label{mid:eq:1}=&-2\big(1-\theta(u)\big)d \cu_- d\cv_- -2\theta(u)d\cu_+ d\cv_+ -\delta(u) (H-v)\big(d\cu_-+d\cu_+\big) du.
\end{align}

We can now extend $d\cu_{\pm}$ arbitrarily (but smoothly) to a neighbourhood of the matching hypersurface, and use the fact that for any two continuous one-forms $\bs{\omega}_1,\bs{\omega}_2\in\Gamma(T^{\star}\M)$ satisfying $\bs{\omega}_1\vert_{u=0}=\bs{\omega}_2\vert_{u=0}$, it holds that $\bs{\omega}_1\cdot \delta(u) =\bs{\omega}_2\cdot \delta(u)$ (since $\delta$ is a distribution of order zero).  
In particular, this allows us to write 
(cf.\ \eqref{eq1a} and \eqref{one-forms:minus:side})
\begin{equation}
\label{mid:eq:2} \delta(u) d\cu_-=\delta(u) du,\quad\delta(u) d\cu_+=\delta(u) \cu_1 du.
\end{equation}

\begin{figure}[t]
\includegraphics[angle=-90,width=1\linewidth]{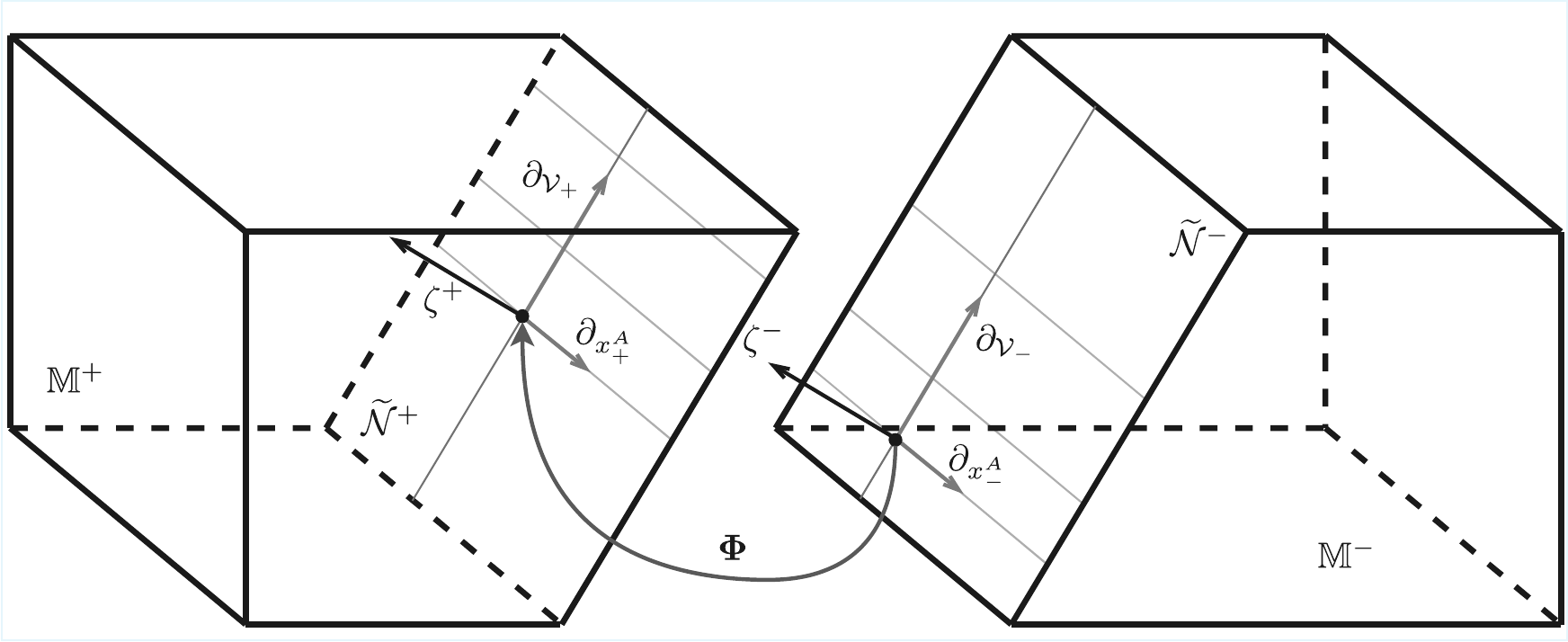}
\caption{The cut-and-paste construction corresponding to the most general matching of two regions of Minkowski space across a null hyperplane.\ Minkowski space is cut along a null hyperplane, and then the two regions are reattached with an arbitrary jump along the null direction of the matching hypersurface.}
\label{fig:CAP}
\end{figure}

By making use of \eqref{mid:eq:2} in the last term of \eqref{mid:eq:1}, one obtains  
\begin{align}
\label{dist:metric:term:2} -2d\cu d\cv=&-2\big(1-\theta(u)\big)d \cu_- d\cv_- -2\theta(u)d\cu_+ d\cv_+ -\delta(u) (H-v)\big(1+\cu_1 \big) du^2.
\end{align}
Following the same reasoning, we can likewise compute the metric term $2\hor(v,z^C) \delta(u) d\cu^2$ explicitly.\ \tcte{In fact, combining  
\eqref{def:hor}, \eqref{dcu:squared}-\eqref{products:2} with the fact that $\hor$ can be rewritten as $\hor(v,z^C)=\frac{1+\cu_1}{1+\cu_1^2}\lp H-v\rp$ (by \eqref{u1v1x1})},
\tcte{one gets}
\begin{align}
\nn 2\hor(v,z^C) \delta(u) d\cu^2=&\spc \frac{2\lp 1+\cu_1\rp}{1+\cu_1^2}\delta(u)(H-v)\Big(\big(1-\theta(u)\big)d\cu_-^2+\theta(u)d\cu_+^2\Big)\\
\label{dist:metric:term:3}=&\spc \frac{ 1+\cu_1}{1+\cu_1^2}\delta(u)(H-v)\Big(d\cu_-^2+d\cu_+^2\Big)
=\delta(u)(H-v)\lp 1+\cu_1\rp du^2,
\end{align}
where in the last step we have used \eqref{mid:eq:2}.\ Finally, inserting \eqref{dist:metric:term:1}, \eqref{dist:metric:term:2} and \eqref{dist:metric:term:3} into \eqref{metric:distributional:Mk}, it is straightforward to get the Lipschitz continuous metric form \eqref{c0metricprevious} of the matched spacetime.

Therefore, using regularisation products, we have established that the Lipschitz continuous metric \eqref{c0metricprevious} transforms to the distributional metric \eqref{metric:distributional:Mk} via the ``discontinuous coordinate transformation'' \eqref{disc:coord:transf:u}-\eqref{disc:coord:transf:x}.
This means that, for \textit{completely general} matchings of two regions of Minkowski space across a null hyperplane, the information of the shell can be entirely encoded in an impulsive, Dirac-$\delta$ term supported on the matching hypersurface.\ In other words, by \eqref{disc:coord:transf:u}-\eqref{metric:distributional:Mk} we have extended the cut-and-paste method to general shells with possibly nonzero energy density, energy flux, and pressure  in Minkowski space. The corresponding ``generalized Penrose conditions" are (cf.\  \eqref{phi:minus:Mk}-\eqref{phi:plus:Mk} 
and \eqref{disc:coord:transf:u}-\eqref{disc:coord:transf:x})
\begin{align}
	\label{gen:penrose:junct:cond} &\cu_+\stackbin{\nullhyp}=\cu_-=0,\qquad 
	\cv_+\stackbin{\nullhyp}= H\big(\cv_-,x_-^A\big),\qquad 
	x_+^A\stackbin{\nullhyp}=x_-^A,
\end{align}
see also Figure \ref{fig:CAP}.

\begin{remark}
As discussed before, in the standard cut-and-paste construction by Penrose et al.\ \textup{\cite{penrose1968twistor,Penrose:1972xrn,podolsky2017penrose}}, the jump function is of the form $H(v,z^A)=v+\mathcal{H}(z^A)$, so that $\cp_vH=\cu_1=1$ and (cf.\ \eqref{disc:coord:transf:u}-\eqref{disc:coord:transf:x})
\begin{align}
\cu= u,\qquad \cv= v+\heav\mathcal{H}(z^A)+\up\cv_1,\qquad \cx^A= z^A+\up x_1^A.
\end{align}
Thus, \eqref{metric:distributional:Mk}, \eqref{gen:penrose:junct:cond} become 
the well-known cut-and-paste distributional metric \eqref{CAP:Mk:metric}, 
as well as the well-known Penrose junction conditions \eqref{penrose:junct:cond1}.\ 
\end{remark}

\begin{remark}
Two aspects of the metric form \eqref{metric:distributional:Mk} deserve attention: (i) the argument of the Dirac delta $\delta$ is the coordinate $u$ rather than $\cu$, and (ii) the impulsive term is ruled by a function $\hor$, which has been defined in terms of the coordinates $\{v,z^A\}$ instead of $\{\cu,\cv,\cx^A\}$.

The reason for (i), i.e.\ for using $u$ as the argument for $\delta$ is that, for generic shells (which may be such that $\cu_1\neq0$, cf.\ \eqref{u1v1x1}), $\cu$ is only Lipschitz continuous at the shell (see \eqref{disc:coord:transf:u}) and hence it is a highly delicate matter to deal with $\delta\lp\cu\rp$. In fact, this necessitates the use of a fully nonlinear distributional geometry, cf.\ \textup{\cite[Ch.\ 3]{grosser2013geometric}}, which is beyond the scope of this article and a topic for future research.\ In analogy with the works on the cut-and-paste formalism \textup{\cite{SSLP:16,SS:17}}, a first step will be a fully nonlinear analysis of the geodesic equation of the metric \eqref{metric:distributional:Mk} with a generic smooth function $\hor(\cu,\cv,\cx^A)$.

Concerning (ii), the scalar function $\hor$ must be well-defined on the matching hypersurface $\nullhyp$,  
and our approach only allows to express it in terms of $\{v, z^A\}$, as otherwise $\hor$ would in general depend on the coordinate $\cv$
which is only defined piecewise almost everywhere, and in particular not on $\nullhyp$. However, it is expected that a nonlinear distributional analysis will also shed some light on the specific form of $\hor$ in terms of the coordinates $\{\cu,\cv,\cx^A\}$.
\end{remark}

\begin{remark}
In general, the function $\hor$ will depend not only on the coordinates $\{\cv,\cx^A\}$ but also on $\cu$.\ This means that the metric form \eqref{metric:distributional:Mk} no longer belongs to the class of $pp$-wave spacetimes, but to the more general Kundt class \textup{\cite[Ch.\ 18]{griffiths2009exact}}.
\end{remark}

\section{Matching of two regions of (anti-)de Sitter spacetime}\label{sec:AdS}

In this section we extend our results of Sections \ref{sec:lip:metric} and \ref{sec:dist:metric} to the matching of two regions of a constant-curvature spacetime across a totally geodesic null hypersurface. The generalization will rely on the existence of a suitable null, geodesic vector field $\hat\xi$ on the spacetime after the shell, which is also a rigging at the matching hypersurface. The key aspect of the construction is to ensure that second $\hat\xi$-derivatives of certain coordinates vanish (cf.\ \eqref{second:Der:coords}), 
and in this way obtain a linear coordinate transformation analogous to \eqref{coord:transf:general}. In the following lemma we construct such a vector field in the more general case of a conformally flat spacetime.
\begin{lemma}\label{lem:geodesic:vector}
Consider a conformally flat spacetime $(\hat{\M},\hat{g})$, i.e. 
	\begin{equation}
		\hat{g}=\frac{\eta_{\mu\nu}d\hat\cx^{\mu}d\hat\cx^{\nu}}{\Omega^2},
	\end{equation}
where $\eta$ is the flat metric, $\{\hat\cx^{\alpha}\}$ is a coordinate system on $(\hat\M,\hat g)$ and $\Omega\in\Fcal^{\star}(\M)$.\ 
Let $\hat\N\subset\hat\M$ be an embedded null hypersurface, $\hat\xi$ a null rigging  along $\hat\N$, and $\hat\nabla$ the Levi-Civita covariant derivative of $\hat g $.\ Extend $\hat\xi$ uniquely to a suitable neighbourhood $\hat{\mathcal{O}}$ of $\hat\N$ 
by enforcing 
	\begin{align}
		\label{def:function:F}\hat\nabla_{\hat\xi}\hat\xi=-F\hat\xi,\qquad 
		F\defi \frac{2d\Omega(\hat\xi)}{\Omega}.\ 
	\end{align}
	Then $\hat\xi$ is null on $(\hat{\mathcal{O}},\hat g\vert_{\hat{\mathcal{O}}})$, and
	$\hat\xi\big(\hat\xi\big(\hat\cx^{\alpha}\big)\big)\vert_{\hat{\mathcal{O}}}=0$.\  
\end{lemma}
\begin{proof}
First, we prove
that $\hat\xi$ is everywhere null in $\hat{\mathcal{O}}$.\  
Indeed, 
$\hat\nabla_{\hat\xi}\big(\hat g(\hat\xi,\hat\xi)\big)=2\hat g\big(\hat\nabla_{\hat\xi}\hat\xi,\hat\xi\big)=-2F\hat g(\hat\xi,\hat\xi)$,
which together with $ \hat g(\hat\xi,\hat\xi)\vert_{\hat\N}=0$ implies, by uniqueness of solutions of ordinary differential equations, that 
$\hat g(\hat\xi,\hat\xi)=0$ vanishes on $(\hat{\mathcal{O}},\hat g\vert_{\hat{\mathcal{O}}})$.\ A direct calculation then yields
\begin{align}
\nn \hat\xi(\hat\xi(\hat\cx^{\nu}))
&=(\hat\xi^{\alpha}\hat\nabla_{\alpha}\hat\xi^{\beta})\hat\nabla_{\beta}\hat\cx^{\nu}+\hat\xi^{\alpha}\hat\xi^{\beta}\hat\nabla_{\alpha}\hat\nabla_{\beta}\hat\cx^{\nu}
=-F\hat\xi^{\beta}\hat\nabla_{\beta}\hat\cx^{\nu}+\hat\xi^{\alpha}\hat\xi^{\beta}\hat\nabla_{\alpha}\hat\nabla_{\beta}\hat\cx^{\nu}\\
\nn &=-F\hat\xi^{\beta}\delta_{\beta}^{\nu}+\hat\xi^{\alpha}\hat\xi^{\beta}( \cp_{\alpha}\cp_{\beta}\hat\cx^{\nu}-\hat\Gamma^{\rho}_{\beta\alpha}\cp_{\rho}\hat\cx^{\nu})
=-F\hat\xi^{\nu}-\hat\xi^{\alpha}\hat\xi^{\beta} \hat\Gamma^{\nu}_{\beta\alpha}\\
\label{linear} &=-F\hat\xi^{\nu}+\frac{2}{\Omega} \hat\xi^{\alpha}(\cp_{\alpha}\Omega)\hat\xi^{\nu}=0 ,
\end{align}
where in the last line we used that 
$\hat\xi$ is null and that the Christoffel symbols of $\hat g$ satisfy 
$\hat\Gamma^{\nu}_{\alpha\beta}	=-\big( (\cp_{\alpha}\log\Omega)\delta_{\beta}^{\nu}+(\cp_{\beta}\log\Omega)\delta_{\alpha}^{\nu}-\hat g^{\nu\rho}(\cp_{\rho}\log\Omega)\hat g_{\alpha\beta}\big)$.\  
\end{proof}
Now, consider two regions $(\hat \M^{\pm},\hat g^{\pm})$ of a constant-curvature spacetime.\ 
It is well-known \cite{ferrandez2001geometry,navarro2016null} that $(\hat \M^{\pm},\hat g^{\pm})$ admit---up to isometries---only one totally geodesic null hypersurface.\ We let $\hat\N^{\pm}\subset \hat \M^{\pm}$ be two such hypersurfaces.\ Then, there exist coordinates $\{\cu_{\pm},\cv_{\pm},x_{\pm}^A\}$ adapted to $\hat\N^{\pm}$ where the metric $\hat g^{\pm}$ is conformally flat and $\hat\N^{\pm}=\{\cu_{\pm}=0\}$, i.e.\ 
\begin{align}
\nn \hat g^{\pm}&=\frac{\eta^{\pm}}{\Omega^2_{\pm}},\quad\flatpm=-2d\cu_{\pm} d \cv_{\pm}+\delta_{AB} dx_{\pm}^A dx_{\pm}^{B},\quad \cu_{+}\geq0,\quad \cu_{-} \leq 0,\\
\label{CF:spacetimes:1}	\Omega_{\pm}&\defi 1+\frac{\Lambda}{12}\lp \delta_{AB}x_{\pm}^Ax_{\pm}^B-2\cu_{\pm}\cv_{\pm}\rp.
\end{align}
Here $\Lambda$ denotes the cosmological constant, so that Minkowski, de Sitter and anti-de Sitter spacetimes correspond to $\Lambda=0$, $\Lambda>0$ and $\Lambda<0$, respectively.\ The coordinates $\{\cu_{\pm},\cv_{\pm},x^A_{\pm}\}$ cover a one-sided neighbourhood of the entire hypersurface $\hat\N^{\pm}$ when $\Lambda\leq0$, while for $\Lambda>0$ one generator of $\hat\N^{\pm}$ must be removed because $\hat\N^{\pm}\cong\mathbb{S}^{\n-1}\times\mathbb{R}$ and stereographic coordinates only cover the sphere minus one point. 

The induced metrics on $\hat\N^{\pm}$ are given by 
\begin{align}
\hat g^{\pm}\vert_{\hat\N^{\pm}}
=\Omega^{-2}_{\pm}\vert_{\hat\N^{\pm}}\delta_{AB} dx_{\pm}^A dx_{\pm}^{B}, 
\end{align}
and $\hat\N^{\pm}$ has product structure $S^{\pm}\times\mathbb{R}$ where $S^{\pm}\defi \{\cu_{\pm}=0, \cv_{\pm}=0\}$ is a spacelike cross-section and the null generators are along $\mathbb{R}$.\ 
Following a procedure analogous to the one in Section \ref{sec:intro:MK}, we choose  
$\{L^{\pm},k^{\pm},v_I^{\pm}\}$ to be 
\begin{equation}
	\label{basis:L,k,v:CF} 
	L^{\pm}=-\Omega^2_{\pm}\cp_{\cu_{\pm}},\qquad k^{\pm}=\cp_{\cv_{\pm}},\qquad v^{\pm}_I=\cp_{x^I_{\pm}},
\end{equation}
and construct metric data embedded in $\hat\M^-$ with the embedding and the rigging
\begin{align}
	\label{phi:minus:CF}
	\phi^-(v, z^I) =\left(\cu_-=0,\cv_-= v, x_-^I=z^I \right ):\N\longhookrightarrow\hat\N^-\subset\hat\M^-,\qquad \rig^-=L^-.
\end{align}

The feasibility of the matching relies on the solvability of 
\eqref{junct1}, which now reads (cf.\ \eqref{embedPhi+})
\begin{align}
	\label{jcCFcase}
	\Omega^{-2}_{-}\vert_{\{\cu_-=0\}}\delta_{IJ}\vert_{\phi^-(p)} &= \Omega^{-2}_{+}\vert_{\{\cu_+=0\}}\frac{\partial h^L}{\partial z^I}\frac{\partial h^K}{\partial z^J}\delta_{LK}\vert_{\phi^+(p)},\quad \forall p\in \N,\\
\label{phi:plus:CF:1}
		\textup{where}\quad \phi^+(v, z^I) &=\left(\cu_{+} =0, \cv_{+} = H(v, z^I), x_{+}^I = h^I (z^J) \right ).
\end{align}

This is an isometry condition between the cross-sections 
$\phi^{\pm}(\{v=\textup{const.}\})\subset\N^{\pm}$ and, as for Minkowski space,  it is identically satisfied when the identification between the null generators of $\hat\N^{\pm}$ is trivial, i.e.\ when $h^A=z^A$.\ In this case, the conformal factors coincide on $\hat\N^{\pm}$, namely 
$\Omega_{\pm}\vert_{\{\cu_{\pm}=0\}}=1+\frac{\Lambda}{12} \delta_{AB}z^Az^B\defi \Omega_{\N}$, and one can prove that 
\begin{equation}
\label{MHD:CF}\Big\{\N,\spc\gamma=\frac{\delta_{AB}}{\Omega^2_{\N}}dz^A\otimes dz^B,\spc\ellc=dv,\spc \elltwo=0\Big\}
\end{equation}
is null metric 
data $\{\phi^-,\rig^-\}$-embedded in $\hat\M^-$, and 
embedded 
in $\hat\M^+$ with (cf.\ \eqref{uniqxi}) 
\begin{align}
\label{phi:plus:CF}\phi^+(v, z^I)& =\left(\cu_{+} =0, \cv_{+} = H(v, z^I), x_{+}^I =  z^I \right ),\quad\textup{and}\\
\label{uniqxi:CF}
\rig^+&=-\Omega_{\N}^2\lp\frac{\cp H}{\cp v}\rp^{-1}\lp 
\cp_{\cu_{+}}	+  \delta^{AB} \frac{\cp H}{\cp z^A}
\lp \frac{1}{2} 
\frac{\cp H}{\cp z^B}     \cp_{\cv_{+}}+\cp_{x^B_{+}}\rp \rp.
\end{align}
Observe that, as in the case of Minkowski space (cf.\ Section \ref{sec:intro:MK}), $\rig^-$ points inwards 
and $\rig^+$ points outwards, which means that $\rig^{\pm}$
satisfy the orientation condition  \hyperlink{junc-cond-FHD}{$(ii)$} and hence that  the matching of $(\hat\M^{\pm},\hat g^{\pm})$ across $\hat\N^{\pm}$ is feasible for any\footnote{Recall that when the boundaries are totally geodesic, infinitely many matchings are possible \cite{manzano2022general,manzano2023matching}.} jump function $H(v,z^A)$.\ 
We call $(\hat\M,\hat g)$ the corresponding matched spacetime.

Next we explicitly construct coordinates $\{u,v,z^A\}$ in which the metric $\hat g$ of the matched spacetime is Lipschitz continuous.\ As in Section \ref{sec:lip:metric}, we enforce a trivial identification of coordinates on $(\hat\M^-,\hat g^-)$, i.e.\ 
\begin{equation}
\ld\{\cu_-=u,\spc\cv_-=v,\spc x^A_-=z^A\}\rv_{\hat\M^-}.
\end{equation}
This, in combination with \eqref{eiyl} and \eqref{basis:L,k,v:CF}, entails (recall that the two bases $\{\rig^{\pm},e_a^{\pm}\}$ are identified in the process of matching)
\begin{equation}
\label{third:equality}	e^{\pm}_1=\cp_{v},\qquad e^{\pm}_I=\cp_{z^I},\qquad \Omega_{\N}^{-2}\rig^{\pm}=-\cp_{u}.
\end{equation}
Now, $\rig^{\pm}$ are both null riggings along $\hat\N^{\pm}$, 
hence the coordinate vector $\cp_u$ is also a null rigging therein.\ 
We can therefore use Lemma \ref{lem:geodesic:vector} and extend $\cp_u$ uniquely to a null vector field $\hat\xi=\cp_u$ on $(\hat\M^+,\hat g^+)$ so that $\hat\xi(\hat\xi(\cu_+))=\hat\xi(\hat\xi(\cv_+))=\hat\xi(\hat\xi(x^A_+))=0$.\ Accordingly, the relation between the coordinate systems $\{\cu_+,\cv_+,x_+^A\}$ and $\{u,v,z^A\}$ is of the form \eqref{coord:transf:general}.\ This, together with  \eqref{phi:plus:CF}-\eqref{uniqxi:CF} and the third equality in \eqref{third:equality}, allows us to fix the functions $\cu_0,\cu_1,\cv_0,\cv_1,x_0,x_1^A$ and conclude that the coordinate transformation in $(\hat\M^+,\hat g^+)$ is given by \eqref{eq1a:0} and \eqref{u1v1x1}, just as in the case of Minkowski space. Summing up we have
\begin{align}
\label{one:more:eq:1}	\cu_-&=u  , & \cv_-&=v, & x_-^A&=z^A,\\
\label{one:more:eq:2}	\cu_+&=u  \hspace{0.075cm}\cu_1, & \cv_+&=H+u \cv_1, & x_+^A&=z^A+u  x_1^A,
\end{align}
where $\cu_1,\cv_1,x_1^A$ are defined by \eqref{u1v1x1}. The computations to derive the corresponding one-forms $\{d\cu_{\pm},d\cv_{\pm},dx_{\pm}^A\}$ and the Lipschitz metric of the matched spacetime are therefore the same as in Section \ref{sec:lip:metric}. Indeed, by reproducing the calculations therein, 
it is immediate to prove that the metrics $\hat g^{\pm}$ read 
\begin{align}
\nn \hat g^{\pm}&=\frac{\eta^{\pm}}{\Omega_{\pm}^2},\qquad\textup{where}\\
\Omega_{-}&= 1+\frac{\Lambda}{12}\lp \delta_{AB}z^Az^B-2uv\rp,\\
\nn \Omega_{+}&= 1+\frac{\Lambda}{12}\lp \delta_{AB}z^Az^B-2uv\rp+\frac{u\Lambda}{6}\lp \delta_{AB}\lp  z^Ax_1^B+\frac{u}{2}  x_1^A  x_1^B\rp-\cu_1\lp   H+u  \cv_1\rp+ v\rp,
\end{align}
and $\eta^{\pm}$ are given by \eqref{flatm:uvz} and \eqref{eq5}. Observe that $\flatp$ is written in terms of the tensor $[\bY]$ and the pressure of the Minkowski case (see \eqref{YpmMk:and:tauMk}-\eqref{def:eg:flux:pressure}), which a priori do \textit{not} need to coincide with the tensor $[\hat\bY]$ and the pressure $\hat p$ corresponding to the matching of $(\hat \M^{\pm},\hat g^{\pm})$.\ However, a straightforward calculation based on Proposition \ref{prop6} yields
\begin{align}
\nn [\hat \Y_{vv}]&=[\Y_{vv}],\qquad [\hat \Y_{vz^A}]=[\Y_{vz^A}],\\
[\hat \Y_{z^Az^B}]&=[\Y_{z^Az^B}]+\frac{\Lambda \delta_{AB}}{6\Omega_{\N}(\cp_vH)}\Big( v(\cp_vH)-H+z^C(\cp_{z^C}H)\Big).
\end{align}
Consequently, only the spatial components of $[\hat \bY]$ differ from those of $[\bY]$, and in particular the pressure $\hat p\defi [\hat \Y_{vv}] $ for arbitrary $\Lambda$ \cite[Rem.\ 2.16]{manzano2023matching} and the pressure $p$ for the Minkowski case coincide.\ Observe also that $[\hat\bY]=[\bY]$ if $\Lambda=0$, which represents a non-trivial consistency check of the previous results.

By defining the Lipschitz 
function 
\begin{align*}
Q\defi \lb 
\begin{array}{ll}
1+\frac{\Lambda}{12}\lp \delta_{AB}z^Az^B-2uv\rp & u<0\\
1+\frac{\Lambda}{12}\lp \delta_{AB}z^Az^B-2uv\rp+\frac{u\Lambda}{6}\lp \delta_{AB}\lp  z^Ax_1^B+\frac{u}{2}  x_1^A  x_1^B\rp-\cu_1\lp   H+u  \cv_1\rp+ v\rp & u\geq0
\end{array}\rd
,
\end{align*}
it is immediate to check that 
the metric $\hat g$ of the matched spacetime
is given by
\begin{align}
\label{C0:metric:CF}\hat g&
=\frac{g}{Q^2},
\end{align}
where $g$ is the Lipschitz continuous metric \eqref{c0metricprevious}.
Since both $Q$ and $g$ are Lipschitz continuous, it follows that $\hat g$ is Lipschitz continuous as well. 

Finally we also derive the distributional form of the metric \eqref{C0:metric:CF}. Introducing coordinates $\{\cu,\cv,\cx^A\}$ exactly as in the flat case, i.e.\ \eqref{disc:coord:transf:u}--\eqref{disc:coord:transf:x} and using \eqref{metric:distributional:Mk}-\eqref{def:hor} and \eqref{C0:metric:CF}, it is immediate to check that $\hat g$ takes the  distributional form 
\begin{equation}
	\label{metric:distributional:CF}
	\hat g=\frac{-2d\cu d\cv+\delta_{AB}d\cx^A \cx^B+2\tcte{\hor(v,z^C)\delta(u)} d\cu^2}{\lp 1+\frac{\Lambda}{12}\lp \delta_{AB}\cx^A\cx^B-2\cu\cv\rp\rp^2}.
\end{equation}
Furthermore, it is also straightforward to conclude that the generalized Penrose junction conditions for the present case 
are again \eqref{gen:penrose:junct:cond}.\ 
In the case of purely gravitational and null-dust shells, it also holds that the Penrose junction conditions remain the same even with non-zero cosmological constant (see \cite{podolsky2019cut,PS:22,podolsky2017penrose,samann2023cut}), hence our conclusions are consistent with previous analysis in the literature.\ Consequently, the results of this section generalize the cut-and-paste construction method for 
spacetimes with non-zero constant curvature 
to (totally geodesic) null shells with \textit{completely general} matter contents, and to \textit{any} jump function $H(v,z^A)$.\

\section{A null shell with pressure and energy flux in Min\-kows\-ki space}\label{sec:ex:pressure:Mk:new}

We conclude the paper with an explicit example of a null shell with non-vanishing energy density $\rho$, energy flux $j$, and pressure $p$ (cf.\ \eqref{def:eg:flux:pressure}) in Minkowski space (hence in the case $\Lambda=0$).  
For simplicity, we restrict to the case where $(\mathbb{M}^{\pm}, \eta^{\pm})$ are $4$-dimensional spacetimes (i.e.\ $\n = 3$).\ As discussed before, the matching information is entirely encoded in the jump function $H$, which can be freely chosen as long as condition $\cp_vH>0$ is fulfiled.\ It is also worth emphasizing that the so-called Israel (or shell) equations \cite{barrabes1991thin, israel1966singular, lanczos1922bemerkung, lanczos1924flachenhafte,   mars1993geometry, senovilla2018equations} are satisfied for \textit{any} jump function\footnote{See \cite[Sect.\ 5.1]{manzano2022general} for a detailed discussion in the case of the matching across Killing horizons with bifurcations surfaces, or \cite[Eq.\  (6.9)]{manzano2021null}.}, so in particular they will hold for our choice of $H$ below.\ Now, consider the jump function 
\begin{align} 
\label{step:funct:example:final}
H(v,z^A)&=a v-b\log\lp\cosh\lp v\rp\rp-c\tanh( r) e^{-v^2}-\frac{\mathcal{H}_0r^2}{4}  \textup{erf}\lp r\rp,
\end{align}
where $r \defi \sqrt{(z^2)^2+(z^3)^2}$ 
and $a,b,c,\mathcal{H}_0\in\mathbb{R}$ are real numbers satisfying:\
\begin{align}
\label{constraints:constants}b\geq 2 c\geq0, \qquad 
\mathcal{H}_0\geq c,\qquad 
a>b+c. 
\end{align}

\begin{figure}[t!]
	\centering
	\includegraphics[width=0.48\linewidth]{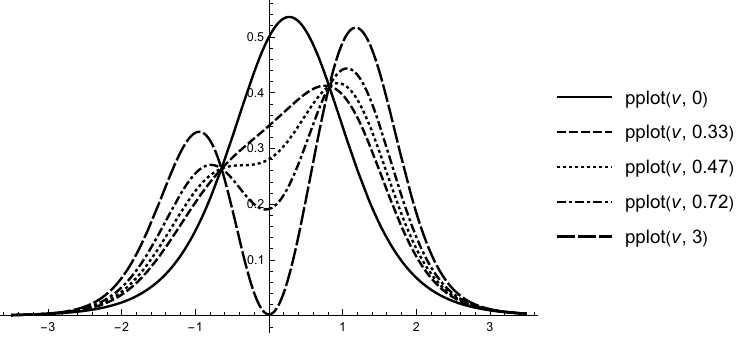} 
	\includegraphics[width=0.48\linewidth]{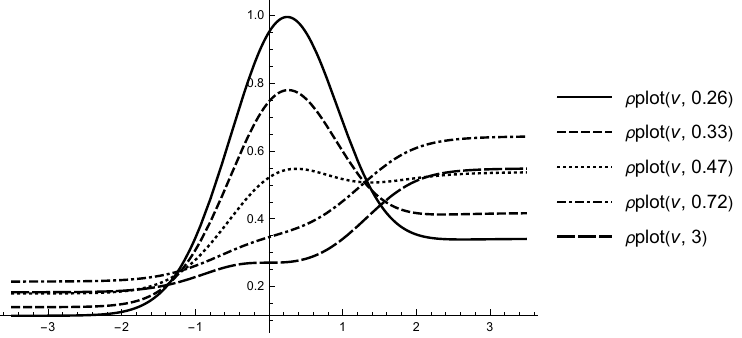}
        \includegraphics[width=0.48\linewidth]{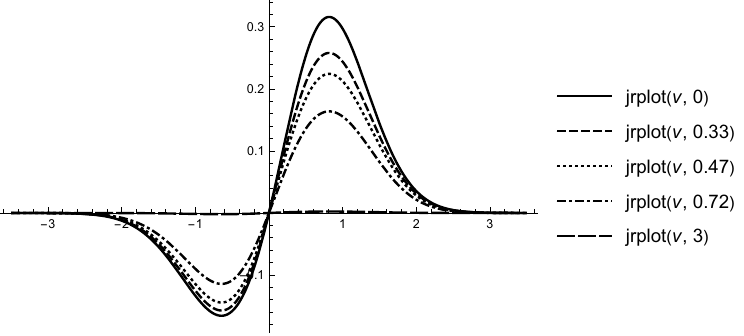}
                \includegraphics[width=0.48\linewidth]{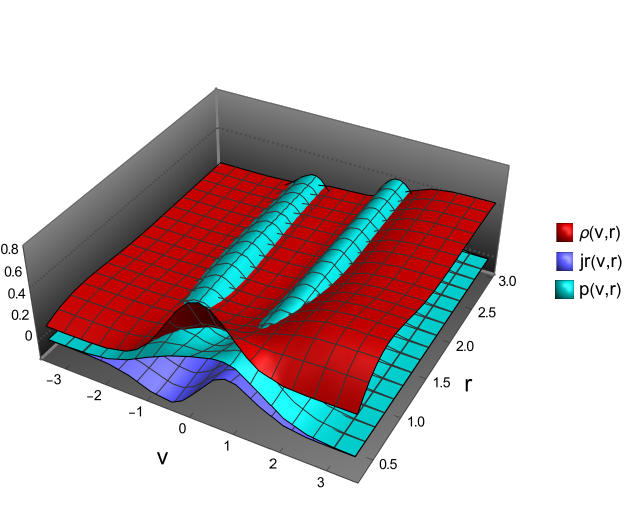}
	\caption{For the values $a=4$, $b=2$, $c=1$ and $\mathcal{H}_0=1.1$ satisfying \eqref{constraints:constants}, plots of the pressure $p(v,r)$, the energy density $\rho(v,r)$, and the radial component of the energy flux $j^r(v,r)\defi r^{-1}\big( z^2j^{z^2} +z^3j^{z^3}\big)$ as functions of $v$ for the values of $r$ indicated in the legends.\ 
    The coloured picture shows a $3$D-plot of $\rho(v,r)$, $j^r(v,r)$ and $p(v,r)$ for $v\in(-3,3), r\in(0.3,3)$, also in the case $a=4$, $b=2$, $c=1$ and $\mathcal{H}_0=1.1$.}
\label{fig:example}
\end{figure}

A direct calculation then yields
\begin{align}
\label{partial:H:ex}\cp_vH =&\spc a-b \tanh(v)+2cv\tanh( r) e^{-v^2},\\
\label{p:ex}p=&\spc \frac{1}{\cp_vH}\bigg( \frac{b}{\cosh^{2}(v)}+2 c\tanh( r)  e^{-v^2} \lp 2v^2-1\rp\bigg),\\
\label{rho:ex}\rho=& \spc \frac{1}{\cp_vH}\lp \frac{c e^{-v^2}}{\cosh^2(r)}\lp \frac{1}{r}-2\tanh(r)\rp+\mathcal{H}_0\lp \textup{erf}(r)+\frac{r e^{-r^2}}{\sqrt{\pi}} \lp\frac{5}{2}-r^2\rp\rp  \rp,\\
\label{jr:ex}j^r=&\spc \frac{2c v  e^{-v^2}}{(\cp_vH)\cosh^2(r)},
\end{align}
where $j^r$ is the radial component of the energy flux, namely its only non-zero component in the present case.\  
From \eqref{constraints:constants}-\eqref{jr:ex}, it is straightforward to check that $\cp_v H > 0$, and that the pressure and the energy density of the shell are no-where negative, i.e.\  $p(v,z^A)\geq0$ and $\rho(v,z^A)\geq0$.  
Figure \ref{fig:example} shows the profiles of the pressure, energy density, and energy flux for some values of the constants $a,b,c,\mathcal{H}_0$ as functions of $v$ for the indicated $r$-values. The pressure remains negligible for large values of $\vert v \vert$, and becomes non-zero only near $v=0$.\ Remarkably, variations in the energy density and energy flux occur precisely where the pressure is not negligible, while away from this region both $\rho\vert_{r\neq0}$ and $j^r$ remain approximately constant.\ In particular, 
$\rho\vert_{r\neq0}$ is nearly zero for large negative values of $v$, and asymptotically approaches a positive constant once the effect of the pressure has already taken place. A similar behaviour was also observed in \cite[Sect.\ 6]{manzano2021null}, and suggests that a positive pressure contributes to an overall increase in the energy density of the shell.\ For $r=0$, on the other hand, the energy density diverges (i.e.\ $\lim_{r\rightarrow0}\rho(v,r)=+\infty$), which physically is compatible with the existence of a particle moving at the speed of light along the null generator $\sigma\defi\{z^2=0,z^3=0\}\subset\nullhyp$, cf.\ the famous Aichelburg-Sexl solution \cite{AS:71}.  Regarding the energy flux vector field $j$, it is oriented towards lower (resp.\ higher) values of the radius $r$ when $v<0$ (resp.\ $v>0$), and it vanishes at $v=0$.\ Moreover, it decays to zero both in the limits $v\rightarrow\pm\infty$ and $r\rightarrow\infty$.\ Thus, the energy flows towards $r=0$ prior to $v=0$ and outwards thereafter.

By particularizing \eqref{c0metricprevious} and \eqref{metric:distributional:Mk} to the present case, one obtains the explicit expressions for the Lipschitz continuous and the distributional metric forms of the matched spacetime. These expressions are rather involved in the case with arbitrary values of $a$, $b$, $c$ and $\mathcal{H}_0$ satisfying \eqref{constraints:constants}. For this reason, we next present them in a simpler setting with non-zero pressure, but vanishing energy density and energy flux.\ 
\begin{example}
Let us fix $a=2$, $b=1$, $c=\mathcal{H}_0=0$, so that $H(v,z^B)=a v-b\log\lp\cosh\lp v\rp\rp $ (cf.\ \eqref{step:funct:example:final}). Then $[\Y_{vz^A}]=[\Y_{z^Az^B}]=0$ (cf.\ \eqref{YpmMk:and:tauMk}), $j^A=\rho=0$ (cf.\ \eqref{partial:H:ex}-\eqref{jr:ex}), and the Lipschitz continuous and the distributional metric forms \eqref{c0metricprevious}, \eqref{metric:distributional:Mk} read
\begin{align*}
g&=-2dudv +\delta_{AB}dz^Adz^B-2u_+( u) \cosh^{-2}(v)\lp 2-\tanh (v)\rp^{-1}dv^2,\\
g&=-2d\cu d\cv+\delta_{AB}d\cx^A \cx^B+\frac{2\big(\tanh (v)-3\big) \big(\tanh (v)-2\big) }{\big(\tanh (v)-4\big) \tanh (v)+5}\Big(v-\log \big(\cosh (v)\big)\Big)\delta(u) d\cu^2.
\end{align*}
\end{example}
It is by virtue of the formalism of matching developed in \cite{manzano2021null,manzano2022general,manzano2023matching}, and in particular as an application of the formalism of hypersurface data \cite{mars2013constraint,mars2020hypersurface}, that the previous examples of a shell with a highly non-trivial matter contents along a null hyperplane in  Minkowski space has been possible to analyze.\ Constructions of this type, which are clearly interesting both from a geometric and from physical viewpoint, are typically disregarded in the literature due to the lack of a sufficiently general formalism of matching that also allows for very explicit calculations.\ It is precisely this flexibility and computational applicability that make the matching framework employed here such a useful tool for the systematic study of null shells.

\appendix

\section{Abstract matching of semi-Riemannian manifolds}\label{app:B}

In this Appendix, we describe the matching of general semi-Riemannian manifolds from the abstract point of view of adjunction spaces. This will enable us to prove that the regularity of the metric of the matched spacetime is locally Lipschitz continuous. This can be seen as a straightening and sharpening of \cite[Sec.\ 3, Prop.]{clarke1987junction}. 

We start by recalling a classical result on the gluing of two smooth manifolds along their boundaries. Note that a \textit{regular domain} in a smooth manifold (with or without boundary) is a smooth codimension-$0$ submanifold with boundary such the embedding map is a proper map, i.e., preimages of compact sets are compact.

\begin{theorem}[Attaching manifolds along their boundaries, Thm.\ 9.29 in \cite {lee2003introduction}]
    Let $M_1$ and $M_2$ be smooth manifolds both of dimension $n$ with nonempty boundaries $\Sigma_i := \partial M_i$, and let $\phi: \Sigma_1 \to \Sigma_2$ be a diffeomorphism. Let $M_1 \cup_{\phi} M_2$ be the adjunction space formed by identifying $x \in \Sigma_1$ with $\phi(x) \in \Sigma_2$\footnote{That is the quotient space of the disjoint union $M_1\dot\cup M_2$ by the equivalence relation $\Sigma_1\ni x_1\sim \phi(x_1)\in \Sigma_2$.}. Then $M_1 \cup_{\phi} M_2$ is a topological manifold without boundary, and has a smooth structure such that there are regular domains $M_1',M_2' \subseteq M_1 \cup_{\phi} M_2$ diffeomorphic to $M_1$ and $M_2$, respectively, satisfying $M_1' \cup M_2' = M_1 \cup_{\phi} M_2$, $M_1' \cap M_2' = \partial M_1' = \partial M_2'$.
\end{theorem}

Next we turn to the issue of including a metric in the construction, that is the gluing of semi-Riemannian manifolds $(M_1,g_1)$ and $(M_2,g_2)$. As it turns out, upgrading the diffeomorphism $\phi$ of the boundaries to an isometry does not quite do the job. More precisely, the existence of an isometry $\phi:\Sigma_1\to\Sigma_2$ is not a sufficient condition to guarantee the existence of a continuous metric $g$ on $M_1 \cup_{\phi} M_2$ that coincides with the original metrics on $M_i$, i.e.\ $g|_{M_i}=g_i$ $(i=1,2$). Since this fact seems not to have been universally appreciated in the literature, we next give a simple counter example.

\begin{example}
    Consider $M_1:=\mathbb{R} \times (-\infty,0]$ with metric $g_1:=dx^2 + dy^2$ and $M_2:=\mathbb{R} \times [0,\infty)$ with metric $g_2:=dx^2 + 2dy^2$. Then $\partial M_1 = \mathbb{R} \times \{0\} = \partial M_2$. Moreover, $id_{\mathbb{R} \times \{0\}}$ is an isometry of the boundaries, and $M_1 \cup_{id} M_2 = \mathbb{R}^2$. But there is no continuous metric on $\mathbb{R}^2$ agreeing with the given metrics on either half space.
\end{example}

This example also hints at the solution of the issue. We have to additionally ask for the isometry to be compatible with the action of the metrics in the transverse directions. We formulate this condition in a way most directly related to the formalism used in the body of this text, cf.\ \eqref{junctcondST:1}--\eqref{junctcondST:3}.

\begin{definition}[$\xi$-aligning isometry]\label{Definition: xialigningisometry}
Let $(M_1,g_1)$ and $(M_2,g_2)$ be smooth semi-Riemannian manifolds with boundary and let $\xi \in C^{\infty}(TM_1|_{\partial M_1})$ be transversal. Let $\phi: \partial M_1 \to \partial M_2$ be a smooth diffeomorphism. Let $M:=M_1 \cup_\phi M_2$ denote the corresponding adjunction space.
We say $\phi$ is a \emph{$\xi$-aligning isometry} if
\begin{enumerate}
        \item[(i)] $\phi: \partial M_1 \to \partial M_2$ is a smooth isometry, i.e.\ 
        $\phi_*(g_1|_{\partial M_1}) = g_2 |_{\partial M_2}$,
    \item[(ii)] $\phi_*(g_1(\xi,\cdot)|_{\Sigma}) = g_2(\xi,\cdot)|_\Sigma\ \in \Omega^1(\Sigma)$,
    \item[(iii)] $\phi_*(g_1(\xi,\xi)) = g_2(\xi,\xi)\ \in C^\infty(\Sigma)$.
   \end{enumerate}
\end{definition}

Note that if we understand $M_1$, $M_2$ (equipped with $g_1$ resp.\ $g_2$) to be subsets of $M$ with common boundary $\Sigma = \partial M_1 = \partial M_2$ then the above conditions can be written in the simplified form
\begin{center}
     (i)  $g_1 |_{\Sigma} = g_2|_{\Sigma}$,\quad
     (ii) $g_1(\xi,\cdot)|_{\Sigma} = g_2(\xi,\cdot)|_\Sigma \in \Omega^1(\Sigma)$,\quad
     (iii) $g_1(\xi,\xi) = g_2(\xi,\xi) \in C^\infty(\Sigma)$.
\end{center}

It is easy to see that if $\phi$ is a $\xi$-aligning isometry and if $\tilde{\xi}$ is any other transversal vector field to $\partial M_1$, then $\phi$ is also $\tilde{\xi}$-aligning. Thus, the definition does not depend on the specific choice of $\xi$, but only on the existence of such a vector field (which is a nontrivial condition on $g_i$ and $\phi$).

\begin{theorem}[Semi-Riemannian matching]\label{Theorem: abstractmatching}
Let $(M_1,g_1)$ and $(M_2,g_2)$ be smooth semi-Rie\-mannian manifolds both 
of dimension $n$ and with the same signature. Suppose both have nonempty boundaries $\partial M_i$ and let $\phi:\partial M_1 \to \partial M_2$ a $\xi$-aligning isometry. Then there exists a unique locally Lipschitz continuous semi-Riemannian metric $g$ of the same index as $g_i$ on the adjunction space $M:=M_1 \cup_\phi M_2$ agreeing with $g_i$ on $M_i$. In particular, $g$ is smooth on $M_i \setminus \Sigma$, $i=1,2$, where $\Sigma$ is the common boundary of $M_1$ and $M_2$ in $M$.
\end{theorem}
\begin{proof}
    Given the identification of $M_i$ with the corresponding subsets of the adjunction space $M$, the property that $\phi$ is $\xi$-aligning simply means that it matches $g_1$ to $g_2$: Indeed, to define $g$ on $M$, we only need to check that $g_1(v,v) = g_2(v,v)$ for every $v \in T_pM$ with $p \in \Sigma$. (The fact that $g_1(v,w) = g_2(v,w)$ for all $v,w \in T_pM$ then follows from polarization.) To see this, let $b_1,\dots, b_n$ be a basis of $T_p \Sigma$. Since $\xi$ is transversal, $\xi(p),b_1,\dots,b_n$ is then a basis of $T_pM$. Writing 
    \begin{equation}
        v = v^0 \xi(p) + \sum_{i=1}^n v^i b_i,
    \end{equation}
    it follows that
    \begin{align*}
        g_1(v,v) = (v^0)^2 g_1(\xi(p),\xi(p)) + 2\sum_{i=1}^n v^0 v^i g_1(\xi(p),b_i) + \sum_{i,j=1}^n v^i v^j g_1(b_i,b_j).
    \end{align*}
    The defining properties of being a $\xi$-aligning isometry precisely say that the terms on the right hand side can be replaced with the corresponding expressions with $g_2$. Hence, $g_1(v,v) = g_2(v,v)$ as claimed.

    From the above we see that $g_1$ and $g_2$ fit at $\Sigma$ to give a continuous semi-Riemannian metric $g$ of the same signature as $g_i$ on $M$. Clearly $g$ is smooth on $M_i \setminus \Sigma$, as it agrees with the smooth metric $g_i$ on that open set. To see that $g$ is locally Lipschitz continuous on all of $M$, observe that its derivatives exist everywhere on $M \setminus \Sigma$ and the limits of the derivatives from either $M_1$ or $M_2$ to $\Sigma$ exist as well, hence $g$ is in $W^{1,\infty}_{\mathrm{loc}}$.
\end{proof}

\begin{remark}[On $\xi$-aligning isometries]
\begin{enumerate}
    \item[]
    \item[(i)] In many cases, there is a natural choice of $\xi$: E.g., if $\partial M \subseteq M$ is Lorentzian or Riemannian, then one would usually choose $\xi$ to be the unit normal vector field.
    \item[(ii)] It is in fact not necessary that the boundary have a transversal vector field to perform the gluing: If there is a collection $\xi_i$ of local transversal vector fields that the isometry $\phi$ maps accordingly as in Definition \ref{Definition: xialigningisometry}, then the proof of Theorem \ref{Theorem: abstractmatching} goes through unchanged. Local transversal vector fields always exist, but the property that an isometry $\phi$ align them appropriately is a nontrivial condition.
    \item[(iii)] If $(M_1,g_1)$ and $(M_2,g_2)$ in Theorem \ref{Theorem: abstractmatching} are spacetimes, i.e., of Lorentzian signature and time orientable), then also $(M,g)$ can be made time orientable by taking a convex combination of the time orientations of $(M_i,g_i)$ across a double collar neighbourhood of $\Sigma$. In the embeddings $M_i \subseteq M$, one would need to possibly invert the time orientation of one of the spacetimes to make this work. In the setting of matching spacetimes with boundary, one often requires that both boundaries have specified transversal vector fields and that the "orientation" of these is preserved by the matching, but this is equivalent to the "matching" of the time orientations as just mentioned. 
    \item[(iv)] Higher regularity of the ``glued metric'' $g$ can be achieved under additional conditions. If $\phi$ preserves the second fundamental forms of $\Sigma$ then $g$ is in $C^{1,1}$ (i.e. its first derivatives are locally Lipschitz) and if $\phi$ preserves the curvature operators on $\Sigma$ it is even $C^{2,1}$, see \cite[Lem.\ 5.1.1]{Spi:16} for details. However, in the context of matching spacetimes we have seen that a jump in the second fundamental form is associated to the matter content of the shell and a regularity higher than Lipschitz continuity is not possible in such a case.
\end{enumerate}
\end{remark}

\section{Regularization products}\label{app:distributional:identities}

In this appendix we calculate the regularisation products used in the computations of Section \ref{sec:dist:metric}. Our method can be placed in the general classification of intrinsic distributional products of \cite[p.\ 69]{Obe:92}, where it is called by the (hardly celebrity) name ``model product (7.4)''. Our main reference for distribution theory is \cite{FJ:98} and we freely use its notation.

To begin with, let $\rho$ be a mollifier, i.e.\ a smooth function $\rho\in C^{\infty}(\mathbb{R}^n)$ with compact support $\supp(\rho)$ in the unit ball $ B_1(0)$ that satisfies $\int\rho(z)dz=1$.\ Consider $\eps\in (0,1]$ and set $\rho_{\eps}(x)\defi \frac{1}{\eps}\rho\lp\frac{x}{\eps}\rp$.\ Then $\rho_{\eps}$ is called a \textit{model-$\delta$-net} and is used to regularize distributions via convolution. More precisely, we regularize an arbitrary distribution $u\in\mathcal{D}'(\mathbb{R}^n)$ via
\begin{equation}\label{eq:conv}
    u*\rho_\eps(x)\defi\langle u(x-y),\rho_\eps(y)\rangle, 
\end{equation}
which is a smooth function (in $x$) which weakly converges to $u$ as $\eps\to 0$. Here $\langle\ ,\ \rangle$ denotes the distributional action (in the variable $y$) and we indeed have $\rho_\eps\to\delta$ and hence $ u*\rho_\eps(x)\to\langle u(x-y),\delta(y)\rangle=u(x)$. In particular, the model-$\delta$ net $\rho_\eps$ itself is a very general regularization of the Dirac $\delta$.

The \textit{model product} $[u\,v]$ of two arbitrary distributions $u$ and $v$ is then defined as the weak limit 
\begin{equation}
    [u\,v]\defi\lim_{\eps\to 0} (u*\rho_\eps)\, (v*\rho_\eps), 
\end{equation} 
provided the limit exists for all model-$\delta$ nets $\rho_\eps$ and is independent of the mollifier $\rho$ chosen.
We will use this product in the following but generally denote it by $uv$. We remark, however, that it is a routine task to extend our reasoning to the more general class of \textit{strict-$\delta$ nets} and the (hence more special) strict product, see \cite[Def.\ 7.1]{Obe:92}.

For reasons explained in Section \ref{sec:dist:metric} it is natural in our context to model the Heaviside function $\theta$ by the prime function of $\rho_\eps$, which is consistent with \eqref{eq:conv}. Indeed we have 
\begin{align}
\theta_{\eps}(x)\defi \theta*\rho_\eps(x)=  \int_{-1}^x\rho_{\eps}(z)dz=\frac{1}{\eps}\int_{-1}^x\rho\lp\frac{z}{\eps}\rp dz=\int_{-1/\eps}^{x/\eps}\rho\lp y\rp  dy=\int_{-1}^{x/\eps}\rho\lp y\rp  dy.\ 
\end{align}
In this case the convergence of $\theta_\eps$ to $\theta$ is even stronger as we have 
\begin{align}
\label{theta:eq}\theta_{\eps}(x)=\int_{-1}^{x/\eps}\rho\lp y\rp  dy=\lb 
\begin{array}{lll}
0 & \textup{if} & x\neq0,\quad x\leq -\eps,\\
1 & \textup{if} & x\neq0,\quad x\geq \eps, \\
\int_{-1}^{0}\rho\lp y\rp  dy & \textup{if} & x=0,
\end{array}\rd
\end{align}
and so $\theta_{\eps}$ converges to $\theta$ pointwise almost everywhere and locally uniformly.\ Observe that, as a direct consequence of \eqref{theta:eq},  
\begin{align}
(\theta_{\eps})^2(x)=\lp\int_{-1}^{x/\eps}\rho\lp y\rp  dy\rp^2=\lb 
\begin{array}{lll}
	0 & \textup{if} & x\neq0,\quad x\leq -\eps,\\
	1 & \textup{if} & x\neq0,\quad x\geq \eps, \\
	\lp \int_{-1}^{0}\rho\lp y\rp  dy\rp^2 & \textup{if} & x=0,
\end{array}\rd
\end{align}
and so $(\theta_{\eps})^2$ also converges to $\theta$ pointwise almost everywhere and locally uniformly. Since we did not specify $\rho$, we have for the model product
$\theta^2=\theta$.

Let us now compute the model product $\theta\delta$. For any test function $\nfi $, we find
\begin{align}
\nn \la \theta_{\eps}\cdot\rho_{\eps},\nfi\ra&=\int_{\mathbb{R}}\theta_{\eps}(z)\rho_{\eps}(z)\nfi(z)dz=\frac{1}{\eps}\int_{\mathbb{R}}\lp \int_{-1}^{z/\eps}\rho\lp y\rp  dy\rp\rho\lp\frac{z}{\eps}\rp\nfi(z)dz\\
\label{app:A:eq0}&=\int_{\mathbb{R}}\lp \int_{-1}^{ t}\rho\lp y\rp  dy\rp\rho\lp t\rp\nfi(\eps t) dt=\int_{-1}^1\lp \int_{-1}^{ t}\rho\lp y\rp  dy\rp\rho\lp t\rp\nfi(\eps t) dt.
\end{align}
Since $\nfi(\eps t)$ uniformly converges to $\nfi(0)$ on $[-1,1]$ as $\eps\to 0$, it follows that 
\begin{align}
\la \theta_{\eps}\cdot\rho_{\eps},\nfi\ra\stackbin[\eps\rightarrow0]{}{\longrightarrow}\nfi(0)\, \mathcal{I},\qquad \textup{where}\qquad \mathcal{I}\defi \int_{-1}^1\lp \int_{-1}^{ t}\rho\lp y\rp  dy\rp\rho\lp t\rp dt.
\end{align} 
We finally calculate $\mathcal{I}$ using integration by parts
\begin{align}
\nn \mathcal{I}&= \int_{-1}^1\lp \int_{-1}^{ t}\rho\lp y\rp  dy\rp\rho\lp t\rp dt=\lc\lp \int_{-1}^{ t}\rho\lp y\rp  dy \rp^2 \rc^1_{-1}-\int_{-1}^1 \lp \int_{-1}^t \rho(s)ds\rp \rho\lp t\rp  dt\\
\label{app:A:eq1}&=\lc \lp \int_{-1}^{ x/\eps}\rho\lp y\rp  dy \rp^2 \rc^{\eps}_{-\eps}-\mathcal{I}=\lc \theta_{\eps}^2(x) \rc^{\eps}_{-\eps}-\mathcal{I}=\lp \theta_{\eps}^2(\eps)-\theta_{\eps}^2(-\eps) \rp-\mathcal{I}=1-\mathcal{I}.
\end{align}
Therefore, $\mathcal{I}=\frac{1}{2}$ and so $\la \theta_{\eps}\cdot\rho_{\eps},\nfi\ra\to \frac{\nfi(0)}{2}$. Since we did not specify the mollifier $\rho$ in all our considerations we have calculated the model product $\theta\delta=\frac{1}{2}\delta$. 

Finally, it is worth emphasizing once more that 
the model product and our results in particular do not depend on the choice of the mollifier $\rho$ so that they are stable with respect to a large class of regularisations. Especially there was no need to suppose that $\theta_{\eps}(0)=\int_{-1}^0\rho(z)dz=\frac{1}{2}$.

\section*{Acknowledgements}
The authors owe special thanks to Marc Mars for many helpful discussions and for support with the calculations and to Ji\v{r}\'{\i} Podolsk\'{y} for his continued encouragement.

This research was funded in part by Ministerio de Ciencia, Innovaci{\'o}n y Universidades [PID2021-122938NB-I00] and [FPU17/03791], and in part by the Austrian Science Fund (FWF) [Grant DOI 10.55776/EFP6 and 10.55776/J4913]. For open access purposes, the authors have applied a CC BY public copyright license to any author accepted manuscript version arising from this submission.

\begingroup
\let\itshape\upshape
\bibliographystyle{acm}
\addcontentsline{toc}{section}{References}
\bibliography{ref}

\begin{thebibliography}{10}

\bibitem{AB:97}
{\sc Aichelburg, P.~C., and Balasin, H.}
\newblock Generalized symmetries of impulsive gravitational waves.
\newblock {\em Classical Quantum Gravity 14}, 1A (1997), A31--A41.
\newblock Geometry and physics.

\bibitem{AS:71}
{\sc Aichelburg, P.~C., and Sexl, R.~U.}
\newblock On the gravitational field of a massless particle.
\newblock {\em General Relativity and Gravitation \textbf{2}\/} (1971),
  303--312.

\bibitem{barrabes2003singular}
{\sc Barrab{\'e}s, C., and Hogan, P.~A.}
\newblock {\em {S}ingular null hypersurfaces in {G}eneral {R}elativity:
  light-like signals from violent astrophysical events}.
\newblock World Scientific, (2003).

\bibitem{barrabes1991thin}
{\sc Barrab{\'e}s, C., and Israel, W.}
\newblock {T}hin shells in general relativity and cosmology: the lightlike
  limit.
\newblock {\em Physical Review D \textbf{43}\/} (1991), 1129--1142.

\bibitem{bonnor1981junction}
{\sc Bonnor, W.~B., and Vickers, P.~A.}
\newblock {J}unction conditions in {G}eneral {R}elativity.
\newblock {\em General Relativity and Gravitation \textbf{13}}, 1 (1981),
  29--36.

\bibitem{CG:12}
{\sc Chru{\'s}ciel, P.~T., and Grant, J. D.~E.}
\newblock On {L}orentzian causality with continuous metrics.
\newblock {\em Classical Quantum Gravity \textbf{29}}, 14 (2012), 145001, 32.

\bibitem{clarke1987junction}
{\sc Clarke, C. J.~S., and Dray, T.}
\newblock {J}unction conditions for null hypersurfaces.
\newblock {\em Classical and Quantum Gravity \textbf{4}\/} (1987), 265.

\bibitem{darmois1927memorial}
{\sc Darmois, G.}
\newblock {L}es {\'e}quations de la gravitation einsteinienne.
\newblock {\em M{\'e}morial des Sciences Math{\'e}matiques, Fascicule XXV
  (Paris: Gauthier-Villars) \textbf{44}\/} (1927).

\bibitem{ferrandez2001geometry}
{\sc Ferr{\'a}ndez, {\'A}., Gim{\'e}nez, {\'A}., and Lucas, P.}
\newblock Geometry of lightlike submanifolds in {L}orentzian space forms.
\newblock In {\em Proc. del Congreso Geometria de Lorentz, Benalmadena\/}
  (2001).

\bibitem{FJ:98}
{\sc Friedlander, F.~G.}
\newblock {\em Introduction to the theory of distributions}, second~ed.
\newblock Cambridge University Press, Cambridge, 1998.
\newblock With additional material by M. Joshi.

\bibitem{geroch1987strings}
{\sc Geroch, R., and Traschen, J.}
\newblock {S}trings and other distributional sources in {G}eneral {R}elativity.
\newblock {\em Physical Review D \textbf{36}}, 4 (1987), 1017.

\bibitem{gourgoulhon20063+}
{\sc Gourgoulhon, E., and Jaramillo, J.~L.}
\newblock {A} $3+1$ perspective on null hypersurfaces and isolated horizons.
\newblock {\em Physics Reports \textbf{423}\/} (2006), 159--294.

\bibitem{GKSS:20}
{\sc Grant, J. D.~E., Kunzinger, M., S\"{a}mann, C., and Steinbauer, R.}
\newblock The future is not always open.
\newblock {\em Letters in Mathematical Physics \textbf{110}}, 1 (2020),
  83--103.

\bibitem{griffiths2009exact}
{\sc Griffiths, J.~B., and Podolsk{\`y}, J.}
\newblock {\em {E}xact space-times in {E}instein's {G}eneral {R}elativity}.
\newblock Cambridge University Press, (2009).

\bibitem{grosser2013geometric}
{\sc Grosser, M., Kunzinger, M., Oberguggenberger, M., and Steinbauer, R.}
\newblock {\em Geometric theory of generalized functions with applications to
  general relativity}, vol.~\textbf{537} of {\em Mathematics and its
  Applications}.
\newblock Kluwer Academic Publishers, Dordrecht, (2001).

\bibitem{israel1966singular}
{\sc Israel, W.}
\newblock {S}ingular hypersurfaces and thin shells in {G}eneral {R}elativity.
\newblock {\em Il Nuovo Cimento B \textbf{44}\/} (1966), 1--14.

\bibitem{khakshournia2023art}
{\sc Khakshournia, S., and Mansouri, R.}
\newblock {\em The Art of Gluing Space-Time Manifolds: Methods and
  Applications}.
\newblock Springer Nature, (2023).

\bibitem{kunduri2013classification}
{\sc Kunduri, H.~K., and Lucietti, J.}
\newblock Classification of near-horizon geometries of extremal black holes.
\newblock {\em Living Reviews in Relativity \textbf{16}\/} (2013), 1--71.

\bibitem{KS:99}
{\sc Kunzinger, M., and Steinbauer, R.}
\newblock A note on the {P}enrose junction conditions.
\newblock {\em Classical Quantum Gravity \textbf{16}}, 4 (1999), 1255--1264.

\bibitem{lanczos1922bemerkung}
{\sc Lanczos, K.}
\newblock {B}emerkung zur de {S}itterschen {W}elt.
\newblock {\em Physikalische Zeitschrift \textbf{23}}, 539-543 (1922), 15.

\bibitem{lanczos1924flachenhafte}
{\sc Lanczos, K.}
\newblock {F}l{\"a}chenhafte {V}erteilung der {M}aterie in der {E}insteinschen
  {G}ravitationstheorie.
\newblock {\em Annalen der Physik \textbf{379}}, 14 (1924), 518--540.

\bibitem{lee2003introduction}
{\sc Lee, J.~M.}
\newblock {I}ntroduction to smooth manifolds.
\newblock {\em Graduate Texts in Mathematics \textbf{218}\/} (2003), 191--194.

\bibitem{lefloch2007definition}
{\sc LeFloch, P.~G., and Mardare, C.}
\newblock Definition and stability of {L}orentzian manifolds with
  distributional curvature.
\newblock {\em Portugaliae Mathematica \textbf{64}}, 4 (2007), 535--573.

\bibitem{manzano2021null}
{\sc Manzano, M., and Mars, M.}
\newblock Null shells: general matching across null boundaries and connection
  with cut-and-paste formalism.
\newblock {\em Classical and Quantum Gravity \textbf{38}}, 15 (2021), 155008.

\bibitem{manzano2022general}
{\sc {M}anzano, M., and {M}ars, M.}
\newblock {G}eneral matching across {K}illing horizons of zero order.
\newblock {\em {P}hysical {R}eview {D} \textbf{106}}, 4 (2022), 044019.

\bibitem{manzano2023matching}
{\sc Manzano, M., and Mars, M.}
\newblock Abstract formulation of the spacetime matching problem and null thin
  shells.
\newblock {\em Physical Review D \textbf{109}\/} (2024), 044050.

\bibitem{manzano2023field}
{\sc Manzano, M., and Mars, M.}
\newblock {N}ull hypersurface data and ambient vector fields: {K}illing
  horizons of order zero and one.
\newblock {\em Physical Review D \textbf{110}\/} (2024), 044070.

\bibitem{manzano2023constraint}
{\sc Manzano, M., and Mars, M.}
\newblock The constraint tensor for null hypersurfaces.
\newblock {\em Journal of Geometry and Physics \textbf{208}\/} (2025), 105375.

\bibitem{mars2013constraint}
{\sc Mars, M.}
\newblock {C}onstraint equations for general hypersurfaces and applications to
  shells.
\newblock {\em General Relativity and Gravitation \textbf{45}\/} (2013),
  2175--2221.

\bibitem{mars2020hypersurface}
{\sc Mars, M.}
\newblock {H}ypersurface data: general properties and {B}irkhoff theorem in
  spherical symmetry.
\newblock {\em Mediterranean Journal of Mathematics \textbf{17}\/} (2020),
  1--45.

\bibitem{mars2024abstract}
{\sc Mars, M.}
\newblock {A}bstract null geometry, energy-momentum map and applications to the
  constraint tensor.
\newblock {\em Beijing Journal of Pure and Applied Mathematics \textbf{1}\/}
  (2024), 797--852.

\bibitem{mars2024transverseI}
{\sc Mars, M., and S{\'a}nchez-P{\'e}rez, G.}
\newblock {T}ransverse expansion of the metric at null hypersurfaces {I}.
  {U}niqueness and application to {K}illing horizons.
\newblock {\em Journal of Geometry and Physics \textbf{209}\/} (2025), 105416.

\bibitem{mars2007lorentzian}
{\sc Mars, M., Senovilla, J., and Vera, R.}
\newblock {L}orentzian and signature changing branes.
\newblock {\em Physical Review D \textbf{76}\/} (2007), 044029.

\bibitem{mars1993geometry}
{\sc Mars, M., and Senovilla, J. M.~M.}
\newblock {G}eometry of general hypersurfaces in spacetime: junction
  conditions.
\newblock {\em Classical and Quantum Gravity \textbf{10}\/} (1993), 1865.

\bibitem{mars2024transverseII}
{\sc Mars, M., and Sánchez-Pérez, G.}
\newblock {T}ransverse expansion of the metric at null hypersurfaces {II}.\
  {E}xistence results and application to {K}illing horizons.
\newblock {\em Journal of Geometry and Physics\/} (2025), 105605.

\bibitem{moncrief1983symmetries}
{\sc Moncrief, V., and Isenberg, J.}
\newblock Symmetries of cosmological {C}auchy horizons.
\newblock {\em Communications in Mathematical Physics \textbf{89}\/} (1983),
  387--413.

\bibitem{navarro2016null}
{\sc Navarro, M., Palmas, O., and Solis, D.~A.}
\newblock Null hypersurfaces in generalized {R}obertson--{W}alker spacetimes.
\newblock {\em Journal of Geometry and Physics \textbf{106}\/} (2016),
  256--267.

\bibitem{Obe:92}
{\sc Oberguggenberger, M.}
\newblock {\em Multiplication of distributions and applications to partial
  differential equations}, vol.~\textbf{259} of {\em Pitman Research Notes in
  Mathematics Series}.
\newblock Longman Scientific \& Technical, Harlow; copublished in the United
  States with John Wiley \& Sons, Inc., New York, 1992.

\bibitem{penrose1968twistor}
{\sc Penrose, R.}
\newblock {T}wistor quantisation and curved space-time.
\newblock {\em International Journal of Theoretical Physics \textbf{1}\/}
  (1968), 61--99.

\bibitem{Penrose:1972xrn}
{\sc Penrose, R.}
\newblock {T}he geometry of impulsive gravitational waves.
\newblock In {\em {{G}eneral {R}elativity: {P}apers in honour of {J}.{L}.
  {S}ynge}}, L.~O'Raifeartaigh, Ed. 1972, pp.~101--115.

\bibitem{podolsky2019cut}
{\sc Podolsk{\`y}, J., S{\"a}mann, C., Steinbauer, R., and {\v{S}}varc, R.}
\newblock {C}ut-and-paste for impulsive gravitational waves with {${\Lambda}$}:
  {T}he geometric picture.
\newblock {\em Physical Review D \textbf{100}\/} (2019), 024040.

\bibitem{PS:22}
{\sc Podolsk\'{y}, J., and Steinbauer, R.}
\newblock Penrose junction conditions with {{$\Lambda$}}: geometric insights
  into low-regularity metrics for impulsive gravitational waves.
\newblock {\em General Relativity Gravitation \textbf{54}}, 9 (2022), Paper No.
  96, 24.

\bibitem{podolsky2017penrose}
{\sc Podolsk{\`y}, J., {\v{S}}varc, R., Steinbauer, R., and S{\"a}mann, C.}
\newblock {P}enrose junction conditions extended: impulsive waves with
  gyratons.
\newblock {\em Physical Review D \textbf{96}\/} (2017), 064043.

\bibitem{PV:98}
{\sc Podolsk\'y, J., and Vesel\'y, K.}
\newblock Continuous coordinates for all impulsive pp-waves.
\newblock {\em Physics Letters A \textbf{241}\/} (1998), 145--147.

\bibitem{poisson2004relativist}
{\sc Poisson, E.}
\newblock {\em {A} relativist's toolkit: the mathematics of black-hole
  mechanics}.
\newblock Cambridge university press, 2004.

\bibitem{samann2023cut}
{\sc S{\"a}mann, C., Schinnerl, B., Steinbauer, R., and {\v{S}}varc, R.}
\newblock Cut-and-paste for impulsive gravitational waves with {$\Lambda$}:
  {T}he mathematical analysis.
\newblock {\em Letters in Mathematical Physics \textbf{114}}, 2 (2024), 58.

\bibitem{SS:17}
{\sc S\"{a}mann, C., and Steinbauer, R.}
\newblock Geodesics in nonexpanding impulsive gravitational waves with
  {$\Lambda$}. {II}.
\newblock {\em J. Math. Phys. \textbf{58}}, 11 (2017), 112503.

\bibitem{SSLP:16}
{\sc S\"{a}mann, C., Steinbauer, R., Lecke, A., and Podolsk\'{y}, J.}
\newblock Geodesics in nonexpanding impulsive gravitational waves with
  {$\Lambda$}, part {I}.
\newblock {\em Classical Quantum Gravity \textbf{33}}, 11 (2016), 115002.

\bibitem{senovilla2018equations}
{\sc Senovilla, J. M.~M.}
\newblock {E}quations for general shells.
\newblock {\em Journal of High Energy Physics \textbf{2018}\/} (2018), 134.

\bibitem{Spi:16}
{\sc Spiegel, F.-M.}
\newblock {\em Scalar curvature rigidity on locally conformally flat manifolds
  with boundary}.
\newblock Ph.{D}. thesis, Rheinische Friedrich-Wilhelms-Universit{\"a}t Bonn,
  Bonn, Germany, August (2016).

\bibitem{S:98}
{\sc Steinbauer, R.}
\newblock On the geometry of impulsive gravitational waves.
\newblock {\em gr-qc/9809054\/} (1998).

\bibitem{SV:09}
{\sc Steinbauer, R., and Vickers, J.~A.}
\newblock On the {G}eroch-{T}raschen class of metrics.
\newblock {\em Classical Quantum Gravity \textbf{26}}, 6 (2009), 065001, 19.

\end{thebibliography}

\end{document}